\documentclass{cedram-aft}

\usepackage{bm}
\usepackage[T1]{fontenc}
\usepackage[english]{babel}
\usepackage[utf8]{inputenc}
\usepackage[mathscr]{euscript}
\usepackage{textcomp,multicol,enumitem}
\usepackage[papersize={180mm,240mm},top=2cm, bottom=2cm, left=2cm , right=2cm]{geometry}

% ...

% ...

\title
[$p$\nobreakdash-Integrality of canonical coordinates]
{$p$\nobreakdash-Integrality of canonical coordinates}

\alttitle{}

\author{\firstname{Daniel} \middlename{} \lastname{vargas-Montoya}}

\email{daniel.vargas-montoya@math.univ-toulouse.fr}

\urladdr{}

\thanks{This work was supported by the National Science Centre of Poland (NCN), grant UMO-2020/39/B/ST1/00940.}

  
\altkeywords{}


\begin{document}
\begin{abstract}
Let $L$ be a differential operator with coefficients in $\mathbb{Q}(z)$ of order $n\geq2$ with maximal unipotent monodromy at zero. In this paper we are interested in determining when the canonical coordinate of $L$ belongs to $\mathbb{Z}_p[[z]]$. For this purpose, motivated by a recent conjecture due to  P. Candelas, X. de la Ossa and D. van Straten~\cite{CD}, we study the situation when $L$ has a strong Frobenius structure $\Phi=(\phi_{i,j})_{1\leq i,j\leq n}\in M_n(\mathbb{Z}_p[[z]])$ such that $\phi_{1,1}(0)=1$. We then give a necessary and sufficient condition for the canonical coordinate of $L$ to belong to $\mathbb{Z}_p[[z]]$ when $L$ has such a strong Frobenius structure.

\end{abstract}


\maketitle


\section{Introduction}

In this paper we study the $p$-integrality of the \emph{canonical coordinate} of  differential operators with \emph{maximal unipotent monodromy} at zero. We remind the reader that $$L=\delta^{n}+a_{n-1}(z)\delta^{n-1}+\cdots+a_1(z)\delta+a_0(z)\in\mathbb{Q}(z)[\delta]\text{, }\quad\delta=z\frac{d}{dz},$$ has maximal unipotent monodromy at zero (MUM type) if, for all $i\in\{0,\ldots,n-1\}$, $a_i(z)\in\mathbb{Q}(z)\cap z\mathbb{Q}[[z]]$. By Frobenius method, we know that if $L$ is of MUM type and $n\geq2$ then there are unique power series $\mathfrak{f}(z)\in1+z\mathbb{Q}[[z]]$ and $\mathfrak{g}(z)\in z\mathbb{Q}[[z]]$ such that
\[y_0=\mathfrak{f}(z)\quad\text{and }\quad y_1=\mathfrak{f}(z)\log z+\mathfrak{g}(z)\]
 are solutions of $L$. 
 
 The canonical coordinate of $L$ is the power series $$q(z):=\exp(y_1/y_0)=z\exp(\mathfrak{g}(z)/\mathfrak{f}(z))=z\left(1+\sum_{j\geq1}\frac{1}{j!}(\mathfrak{g}(z)/\mathfrak{f}(z))^n\right)\in\mathbb{Q}[[z]].$$ This power series is also often called the $q$-coordinate of $L$. We are interested in determining the $p$-integrality of $y_0(z)$ and $q(z)$, that is, we want to know when $y_0(z)$ and $q(z)$ belong to $\mathbb{Z}_p[[z]]$,  where $\mathbb{Z}_p$ is the ring of $p$\nobreakdash-adic integers. Our main source of motivation to study this problem comes from a recent conjecture formulated by P. Candelas, X. de la Ossa and D. van Straten in \cite[Section 4.1]{CD} about \emph{Calabi-Yau differential operators}. A typical example of a Calabi-Yau differential  operator is $$\mathcal{H}=\delta^4-5z(5\delta+1)(5\delta+2)(5\delta+3)(5\delta+4).$$
This differential operator appears in the work of Candelas \textit{et al}~\cite{Ca91}, where they study a mirror family for quintic threefolds in $\mathbb{P}^4$. The differential operator $\mathcal{H}$ satisfies some \emph{algebraic properties}, MUM being one of them, and also satisfies some \emph{arithmetical properties}, namely, $F_0$ and $\exp(F_1/F_0)$ belong to $\mathbb{Z}[[z]]$, where $F_0$ and $F_1$ are solutions of $\mathcal{H}$ given by
$$F_0=1+\sum_{n\geq1}\frac{(5n)!}{(n!)^5}z^n\quad\text{and }F_1=F_0\log(z)+G(z),$$
with $$G(z)=\sum_{n\geq1}\frac{(5n)!}{(n!)^5}(5H_{5n}-5H_n)z^n\quad\text{and } H_n=\sum_{i=1}^n1/i.$$ It is clear that $F_0\in\mathbb{Z}[[z]]$ and it was proven by Lian and Yau in \cite{LY} that $\exp(F_1/F_0)\in \mathbb{Z}[[z]]$. The second fact is very surprising because $G(z)$ is not integral because it has unbounded denominators. In addition, it seems that for many differential operators $L$ of MUM type,  $y_0$ and $\exp(y_1/y_0)$ are $N$-integral\footnote{A power series $f(z)\in\mathbb{Q}[[z]]$ is said to be $N$-integral if there exists a nonzero $N\in\mathbb{Q}$ such that $f(Nz)\in\mathbb{Z}[[z]]$.}. These operators are usally called Calabi-Yau operators. We refer the reader to \cite[Definition 6.5]{BR} for a precise definition. So, a natural question is to determine when a differential operator is of Calabi-Yau type.  Almkvist \textit{et al}~\cite{CYT} have gathered more than 400 differential operators of order 4 which are good candidates to be Calabi-Yau operators. For many differential operators appearing in this list the authors also give the analytic solution $y_0$  and it turns out that for numerous cases $y_0$ belongs to $\mathbb{Z}[[z]]$. Moreover, from Krathentaler and Rivoal~\cite{KT11} and Delaygue~\cite{D13}, we also know that the canonical coordinate of some differential operators of such a list  belongs to $\mathbb{Z}[[z]]$.  
 
 It is also expected that every differential operator appearing in  \cite{CYT} can be obtained as \emph{Picard-Fuchs operator} associated with  families of \emph{Calabi-Yau threfolds} in a one parameter family. 
Following \cite[Theorem 22.2.1]{Kedlaya} or \cite[p.111]{andre}, if $L$ is a Picard\nobreakdash-Fuchs operator then $L$ is equipped with a matrix $\Phi_p=(\phi_{i,j})_{1\leq i,j\leq n}$ with coefficients in $E_p$, called the \emph{strong Frobenius structure}.\footnote{ The field $E_p$ is the field of analytic elements. In Section~\ref{sec_def} we give the definition of $E_p$ and we also give the definition of strong Frobenius structure there.} 
Recently,  P. Candelas, X. de la Ossa and D. van Straten~\cite[Section~4.1]{CD} formulated a conjecture about the strong Frobenius structure for differential operators associated with a family of  Calabi-Yau threfoolds in a one parameter family.
 
 \begin{conj}[\cite{CD}]\label{conj_duco}
 Let $\mathcal{L}\in\mathbb{Q}(z)[\delta]$ be a differential operator associated with a family of Calabi-Yau threfoolds in a one parameter family. Then, for almost every prime number $p$, there exits  a strong Frobenius structure $\Phi_p=(\phi_{i,j})_{1\leq i,j\leq 4}$ for $\mathcal{L}$ such that $\Phi_p\in M_4(\mathbb{Z}_p[[z]])$ and $\phi_{1,1}(0)=1$, $\phi_{1,2}(0)=0=\phi_{1,3}(0)$, and $\phi_{1,4}(0)=p^3\lambda\zeta_p(3)$, where $\zeta_p(3)$ denotes the $p$-adic analog of $\zeta(3)$ and $\lambda$ being a rational number independent of $p$.
 \end{conj}
  As we have already said, it is expected that Conjecture~\ref{conj_duco} to be true for  the differential operators appearing in \cite{CYT} and it is also expected that such differential operators to be Calabi-Yau type. So, an interesting question is to determine if Conjecture~\ref{conj_duco} implies the $N$\nobreakdash-integrality of $y_0(z)$ and $q(z)$. The initial motivation of this paper is to answer a weaker question, namely, if Conjecture~\ref{conj_duco} implies the $p$-integrality of $y_0(z)$ and $q(z)$. On the one hand,  we show that the existence of a strong Frobenius structure $\Phi_p\in M_4(\mathbb{Z}_p[[z]])$ implies $y_0(z)\in1+z\mathbb{Z}_p[[z]]$. On the other hand, under the assumption $\phi_{1,1}(0)=1$, we give a necessary and sufficient condition for $q(z)\in\mathbb{Z}_p[[z]]$. This condition relies on certain properties of a \emph{$p$-integral Frobenius structure} for a differential operator of order $2$.
         
 We would like to point out that,  thanks to Proposition~\ref{prop_frob_p_integral}, if $\mathcal{L}\in\mathbb{Q}(z)[\delta]$ is an irreducible MUM Picard-Fuchs operator of order $4$ then, for almost every prime number $p$, there exists a strong Frobenius structure $\Phi_p=(\phi_{i,j}(z))_{1\leq i,j\leq4}\in M_4(\mathbb{Z}_p[[z]])$ for $\mathcal{L}$ such that $||\Phi_p||=1$.  Since the strong Frobenius structure is unique up to constant (see \cite{Dwork89}), Conjecture~\ref{conj_duco} says that we should have  $\phi_{1,1}(0)=1$. Furthermore, following  Dwork~\cite[Lemma~6.2, (6.16)]{D69}, we know that, for almost every prime number $p$, there is a strong Frobenius structure $\Gamma_p=(\gamma_{i,j}(z))_{1\leq i,j\leq 4}\in M_4(E_p)$ for $\mathcal{L}$ such that $\gamma_{1,1}(0)^2=1$. Since the strong Frobenius structure is unique up to constant, Conjecture~\ref{conj_duco} claims that we should have $\Phi_p=\frac{1}{\gamma_{1,1}(0)}\Gamma_p$.
      
       

\subsection{Main result}

Let us recall that for a differential operator $L=\delta^{n}+a_{n-1}(z)\delta^{n-1}+\cdots+a_1(z)\delta+a_0(z)$ in $\mathbb{Q}(z)[\delta]$ the companion matrix of $L$ is the matrix 	\[A(z)=
\begin{pmatrix}
0 & 1 & 0 & \dots & 0 & 0\\
0 & 0 & 1 & \dots & 0 & 0\\
\vdots & \vdots & \vdots & \vdots & \vdots & \vdots \\
0 & 0 & 0 & \ldots & 0 & 1\\
-a_{0}(z) & -a_{1}(z) & -a_{2}(z) & \ldots & -a_{n-2}(z) & -a_{n-1}(z)\\
\end{pmatrix}.
\]

\begin{defi}(p-integral Frobenius structure)
Let $L$ be a monic differential operator of order $n$ in $\mathbb{Q}(z)[\delta]$ and let $A(z)$ be the companion matrix of $L$.  A $p$\nobreakdash-integral Frobenius structure for $L$ is a matrix $\Phi\in M_n(\mathbb{Z}_p[[z]])$ such that $\det\Phi\neq0$ and
 \begin{equation*}
 \delta\Phi=A(z)\Phi-p\Phi A(z^{p}).
 \end{equation*}
\end{defi}
 
 This definition is inspired by the definition of \emph{strong Frobenius structure} introduced by Dwork in \cite{DworksFf}. The reader can find this definition in Section~\ref{sec_def}. The relation between both definition is given by Proposition~\ref{prop_frob_p_integral}, where it is shown that if $L\in\mathbb{Q}(z)[\delta]$ is MUM type, irreducible and equipped with a Frobenius structure for almost every primer number $p$ then $L$ has a $p$\nobreakdash-integral Frobenius structure for almost every prime number $p$. 
 
In order to state our main result, we need to introduce some notations. For every MUM type differential operator $L\in\mathbb{Q}(z)[\delta]$ of order $n\geq2$ there is associated a MUM type differential operator $L^{(2)}\in\mathbb{Q}_p[[z]][\delta]$ of order $2$ which is uniquely determined by the fact that $y_0=\mathfrak{f}(z)$ and $y_1=\mathfrak{f}(z)\log z+\mathfrak{g}(z)$ are its solutions at zero. In other words, $L^{(2)}=(\delta-t_2)(\delta-t_1)$, where 
\[
t_1=\frac{\delta\mathfrak{f}(z)}{\mathfrak{f}(z)},\quad t_2=\frac{\delta\mathfrak{h}(z)}{\mathfrak{h}(z)}\quad\text{ with }\mathfrak{h}(z)=\mathfrak{f}(z)+\delta\mathfrak{g}(z)-t_1\mathfrak{g}(z).
\]
Finally, we consider the $\mathbb{Q}_p$-linear operator $\Lambda_p:\mathbb{Q}_p[[z]]\rightarrow\mathbb{Q}_p[[z]]$ given by $\Lambda_p(\sum_{n\geq0}a(n)z^n)=\sum_{n\geq0}a(np)z^n.$ This $\mathbb{Q}_p$-linear operator is often called Cartier operator.

We are now ready to state our main result. 
 
\begin{theo}\label{theo_integral_mirror}
Let $L$ be a differential operator  with coefficients in $\mathbb{Q}(z)$ of order $n\geq2$ and of MUM type. Suppose that $L\in\mathbb{Z}_p[[z]][\delta]$ and that $\Phi=(\phi_{i,j}(z))_{1\leq i,j\leq n}$ is a $p$\nobreakdash-integral Frobenius structure for $L$. Then $y_0(z)\in 1+z\mathbb{Z}_p[[z]]$. Moreover, if $|\phi_{1,1}(0)|=1$ then

 \begin{enumerate}
\item the differential operator $L^{(2)}$ belongs to $\mathbb{Z}_p[[z]][\delta]$ and has a $p$-integral Frobenius structure $\Psi=(\psi_{i,j}(z))_{1\leq i,j\leq 2}$ such that $\Psi(0)=\mathrm{diag}(1,p),$
\item the following statements are equivalent:
\begin{enumerate}
\item  $\exp(y_1/y_0)\in\mathbb{Z}_p[[z]]$,\smallskip

\item $y_0(z)=\psi_{1,1}(z)y_0(z^p)\bmod p$,\smallskip

\item $\Lambda_p(\psi_{1,1}(z))y_0(z)=\Lambda_p(y_0(z))\bmod p.$
\end{enumerate}
\end{enumerate}
\end{theo}

Let us make a few comments. 

\textbullet\quad  The condition $L\in\mathbb{Z}_p[[z]][\delta]$ is satisfied for almost every prime number $p$ because $L$ has its coefficients in $\mathbb{Q}(z)\cap\mathbb{Q}[[z]]$.  Moreover, as a consequence of Theorem~\ref{theo_integral_mirror}, we prove in Corollary~\ref{coro_picard_fuchs} that if $L$ is an irreducible Picard-Fuchs equation then $y_0(z)$ belongs to $1+z\mathbb{Z}_p[[z]]$ for almost every prime number $p$.


\textbullet\quad An explicit expression for $\Psi$ is given in the proof of Theorem~\ref{theo_integral_mirror}. This expression depends on the power series $\mathfrak{f}(z)$ and $\mathfrak{g}(z)$.

\textbullet\quad According to Theorem~\ref{prop_arith}, the condition $|\phi_{1,1}(0)|=1$ implies $\exp(y_1/y_0)^{p}\in\mathbb{Z}_p[[z]]$.


\textbullet\quad Thanks to Lucas' Theorem, the analytic solution $y_0(z)$ at zero of many differential operators appearing in \cite{CYT} verifies $\Lambda_p(y_0(z))=y_0(z)\bmod p$ for almost every prime number $p$. We recall that it is hoped that the differential operators appearing in \cite{CYT} to be Calabi\nobreakdash-Yau operators. In particular, it is expected that $\exp(y_1/y_0)\in\mathbb{Z}_p[[z]]$ for almost every prime number $p$. Thus, in view of Theorem~\ref{theo_integral_mirror}, in order to prove  $\exp(y_1/y_0)\in\mathbb{Z}_p[[z]]$, it is sufficient to show that $\Lambda_p(\psi_{1,1}(z))=1\bmod p.$

\textbullet\quad We conjecture that if under the assumptions of Theorem~\ref{theo_integral_mirror} we assume additionally that $\Lambda_p(y_0(z))=y_0(z)\bmod p$ then 
 \[\Psi\bmod p=\begin{pmatrix}
\mathfrak{f}_p(z)\bmod p & 0\\
\delta(\mathfrak{f}_p(z))\bmod p & 0 
\end{pmatrix},
\]
where $\mathfrak{f}_p(z)$ is the $p$-truncation of $\mathfrak{f}(z)$\footnote{The $p$-truncation of a power series $\sum_{n\geq0} c_nz^n$ is the polynomial $\sum_{n=0}^{p-1}c_nz^n$.}. Note that if the conjecture holds then, Theorem~\ref{theo_integral_mirror} implies $\exp(y_1/y_0)\in\mathbb{Z}_p[[z]]$ and that $y_0(z)$ is $p$-Lucas\footnote{A power series $t(z)\in\mathbb{Z}_p[[z]]$ is $p$-Lucas if $t(z)=t_p(z)t(z^p)\bmod p$, where $t_p(z)$ is the $p$-truncation of $t(z)$}.

\subsection{Strategy of the proof}

Let us briefly explain the strategy of the proof of Theorem~\ref{theo_integral_mirror}. The $p$\nobreakdash-integrality of $y_0(z)$ is a direct consequence of Proposition~\ref{prop_p_integral_frob} and Theorem~\ref{theo_generic_integral}. Proposition~\ref{prop_p_integral_frob} is proven in Section~\ref{sec_proofs} and Theorem~\ref{theo_generic_integral} is proven in Section~\ref{sec_p_integral}.  In Section~\ref{sec_p_diff_equa} we use the theory of $p$\nobreakdash-adic differential equations in order to prove some results that are crucial in the proof of Theorem~\ref{theo_generic_integral}. Items (1) and (2) of Theorem~\ref{theo_integral_mirror} are proven in Section~\ref{sec_proof}. For this purpose, we prove in Section~\ref{sec_proof_mirror_maps} some results on $p$-integrality of some power series with coefficients $\mathbb{Q}_p$.

\subsection{Previous results and comparison}
It has been shown, in numerous cases, that the canonical coordinate belongs to $\mathbb{Z}[[z]]$. For example, from the works of Lian and  Yau~\cite[Sec.20, Theorem 5]{LY}, Zudilin~\cite[Theorem~3]{Z02},  Krathentaler and Rivoal~\cite[Theorem 1]{KT10} and Delaygue~\cite{D13}, we know that the canonical coordinate for a large class of MUM hypergeometric operators belongs to  $\mathbb{Z}[[z]]$. Furthermore, Delaygue,  Rivoal and Roques~\cite{D17} gave a characterization of the hypergeometric operators whose canonical coordinate belongs to $\mathbb{Z}[[z]]$.  The works of Krathentaler and Rivoal \cite{KT11}, and Delaygue~\cite{D13} provide examples of  canonical coordinate in $\mathbb{Z}[[z]]$ for some non\nobreakdash-hypergeometric operators. The technique used by Krathentaler and Rivoal, and Delaygue is based on a multi-variate version of Dwork's formal congruences~\cite[Theorem~1]{DworksFfIV}.  In \cite{Volo} there is an approach using $p$\nobreakdash-adic cohomology to prove the integrality of canonical coordinate of Picard\nobreakdash-Fuchs operators coming from families of Calabi-Yau operators. We do not know the real status of \cite{Volo} since it is not published. In contrast, the present work offers a more elementary approach which is based on the theory of $p$\nobreakdash-adic differential equations. The idea of using $p$\nobreakdash-adic tools for proving integrality of canonical coordinate goes back to Stienstra~\cite{S}. Finally, let us mention a result due to Beukers and Vlasenko~\cite{BV21}. Let $L\in\mathbb{Z}_p[[z]][\delta]$ be a differential operator of order $n$. Then, they show that if there exists $\mathcal{A}=\sum_{i=1}^{n}A_i(z)\delta^{i-1}\in\mathbb{Z}_p[[z]][\delta]$ with $A_1(0)=1$ such that, for every solution $y$ of $L$, the composition $\mathcal{A}(y(z^p))$ is a solution of $L$, then $y_0$ and $\exp(y_1/y_0)\in\mathbb{Z}_p[[z]]$. 
Nevertheless, according to Remark~\ref{rem_equivalence}, Theorem~\ref{theo_integral_mirror} implies that  $\exp(y_1/y_0)\in\mathbb{Z}_p[[z]]$ if and only if there exists $\mathcal{B}=B_1(z)+B_2(z)\delta\in\mathbb{Z}_p[[z]][\delta]$ with $B_1(0)=1$ and $B_2(0)=0$ such that, for every solution $y$ of $L^{(2)}$, the composition $\mathcal{B}(y(z^p))$ is a solution of $L^{(2)}$. 


\section{ Frobenius structure and MUM operators}\label{sec_def}

 Let $\mathbb{Q}_p$ be the field of $p$\nobreakdash-adic numbers and $\overline{\mathbb{Q}_p}$ be an algebraic closure of $\mathbb{Q}_p$. It is well\nobreakdash-known that the $p$\nobreakdash-adic norm of $\mathbb{Q}_p$ extends uniquely  to $\overline{\mathbb{Q}_p}.$ Let $\mathbb{C}_p$ be the completion of $\overline{\mathbb{Q}_p}$ with respect to the $p$\nobreakdash-adic norm. The field $\mathbb{C}_p(z)$ is equipped with the Gauss norm which is defined as follows $$\left|\frac{\sum_{i=0}^na_iz^i}{\sum_{j=0}^m b_jz^j}\right|_{\mathcal{G}}=\frac{\sup\{|a_i|\}_{1\leq i\leq n}}{\sup\{|b_j|\}_{1\leq j\leq m}}.$$
The field of \emph{analytic elements}, denoted $E_{p}$, is the completion of $\mathbb{C}_p(z)$ with respect to the Gauss norm. The field $E_{p}$ has a derivation $\delta=zd/dz$. It is easily seen that $E_{p}\subset\mathcal{W}_{p}$, where $$\mathcal{W}_{p}=\left\{\sum_{n\in\mathbb{Z}}a_nz^n: a_n\in\mathbb{C}_p, \lim\limits_{n\rightarrow-\infty}|a_n|=0,\text{ and } \sup\limits_{n\in\mathbb{Z}}|a_n|<\infty \right\}.$$
The ring $\mathcal{W}_{p}$ is usually called the \textit{Amice ring}. This ring is also equipped with the Gauss norm $$\left|\sum_{n\in\mathbb{Z}}a_nz^n\right|_{\mathcal{G}}=\sup\{|a_n|\}_{n\in\mathbb{Z}}.$$ We let $\mathcal{W}_{\mathbb{Q}_p}$ denote the elements of $\mathcal{W}_{p}$ with coefficients in $\mathbb{Q}_p$. Likewise, $\mathcal{W}_{\mathbb{Z}_p}$ is the set of elements of $\mathcal{W}_{p}$ with coefficients in $\mathbb{Z}_p$.  Notice that $\mathbb{Z}_p[[z]]\subset\mathcal{W}_{\mathbb{Z}_p}$. The following definition is due to Dwork~\cite{DworksFf}.
\begin{defi}[strong Frobenius structure]
Let $L$ be a monic differential operator of order $n$ in $\mathbb{Q}(z)[\delta]$ and let $A$ be the companion matrix of $L$. We say that $L$ has a Frobenius structure if there exists $\Phi\in GL_n(E_p)$ such that $$\delta\Phi=A\Phi-\Phi pA(z^{p}).$$
\end{defi}


According to \cite[Theorem 22.2.1]{Kedlaya} or \cite[p.111]{andre}, if $L$ is a Picard\nobreakdash-Fuchs operator then $L$ is equipped with a strong Frobenius structure for almost every prime number.

\begin{defi}[MUM differential operator]
 Let $L$ be a monic differential operator in $\mathbb{Q}_p[[z]][\delta]$. We say that $L$ is MUM if its coefficients belong to  $\mathbb{Q}_p[[z]]\cap z\mathbb{Q}_p[[z]]$.
\end{defi}

\begin{rema}\label{rem_sol}\hfill
\begin{enumerate}

\item  Let $A(z)\in M_n(\mathbb{Q}_p[[z]])$ and let $N=A(0)$. If the eigenvalues of $N$ are all zero then, from~\cite[Chap. III, Proposition 8.5]{Dworkgfunciones}, the differential system $\delta X=A(z)X$ has a fundamental matrix of solutions of the shape $Y_Az^{N}$, where $Y_A\in GL_n(\mathbb{Q}_p[[z]])$, $Y_A(0)=I$ and $z^{N}=\sum_{j\geq0}N^j\frac{(\log z)^j}{j!}$. The matrix $Y_Az^{N}$ will be called \emph{the fundamental matrix of solutions} of $\delta X=A(z)X$ and the matrix $Y_A$ will be called the \emph{uniform part} of solutions of the system $\delta X=AX$.

\item Let $L$ be a MUM differential operator of order $n$ with coefficients in $\mathbb{Q}_p[[z]]$, let $A(z)$ be the companion matrix of $L$ and let $N=A(0)$. Since $L$ is MUM, the eigenvalues of $N$ are all zero. We denote by $X_L= Y_Lz^{N}$ the fundamental matrix of solutions of the system $\delta X=A(z)X$, where the matrix $Y_L$ is the uniform part of the system $\delta X=A(z)X.$  Since the eigenvalues of $N$ are all equal to zero, we have $N^n=0$. For this reason,  $z^{N}=\sum_{j=0}^{n-1}N^j\frac{(\log z)^j}{j!}$.  Consequently, there are unique power series $\mathfrak{f}(z)\in1+z\mathbb{Q}_p[[z]]$ and $\mathfrak{g}\in z\mathbb{Q}_p[[z]]$ such that  $\mathfrak{f}(z)$ and $\mathfrak{f}(z)\log z+\mathfrak{g}(z)$ are solutions of $L$. Moreover, since $\log z$ is  transcendental over $\mathbb{Q}_p[[z]]$, it follows that if $\mathfrak{t}(z)\in \mathbb{Q}_p[[z]]$ is a solution of $L$ then $\mathfrak{t}(z)=c\mathfrak{f}(z)$, where $c=\mathfrak{t}(0)$. 
\end{enumerate}
\end{rema}

\begin{rema}\label{rem_det}
The goal of this remark is to show that if $L\in\mathbb{Q}_p(z)[\delta]$ is a MUM differential operator of order $n$ equipped with a $p$\nobreakdash-integral Frobenius structure given by the matrix $\Phi=(\phi_{i,j})_{1\leq i,j\leq n}$ then $\Phi=Y_L\Phi(0)Y_L(z^{p})^{-1}$ and $\Phi(0)$ is an upper triangular matrix with entries diagonal  $\phi_{1,1}(0),$ $p\phi_{1,1}(0),\ldots, p^{n-1}\phi_{1,1}(0)$. In fact, let $Y_Lz^{N}$ be the fundamental matrix of solutions of $\delta X= A(z)X$, where $N=A(0)$ and $A(z)$ the companion matriz of $L$. Then $Y_L(z^{p})z^{pN}$ is the fundamental matrix of solutions of the system $\delta X=pA(z^{p})X$. We know that \[\delta\Phi=A\Phi-p\Phi A(z^{p})\] and thus, there exists $C\in M_n(\mathbb{C}_p)$ such that we have $\Phi=Y_Lz^{N}C(z^{pN})^{-1}Y_L(z^{p})^{-1}$. Since $$(z^{pN})^{-1}=\mathrm{diag}(1,1/p,\ldots,1/p^{n-1})(z^{N})^{-1}\mathrm{diag}(1,p,\ldots, p^{n-1})$$
and $\log z$ is transcendental over $\mathcal{W}_{p}$, and since moreover $\Phi\in M_n(\mathcal{W}_{\mathbb{Z}_p})$, it follows that   $z^{N}C(z^{pN})^{-1}=C$. In addition, from this equality it is not hard to see that $C$ is an upper triangular matrix with entries diagonal $\mu, p\mu,\ldots, p^{n-1}\mu$ for some $\mu\in\mathbb{C}_p$. Consequently, $\Phi=Y_LCY_L(z^{p})^{-1}$ and $C=\Phi(0)$. Therefore, $\Phi(0)$ is an upper triangular matrix with diagonal given by $\phi_{1,1}(0), p\phi_{1,1}(0), \ldots, p^{n-1}\phi_{1,1}(0).$

\end{rema}

\section{$p$-integrality of $y_0$ and radius of convergence}\label{sec_proofs}

In this section we describe our strategy to prove that $y_0\in 1+z\mathbb{Z}_p[[z]]$. This strategy relies on Proposition~\ref{prop_p_integral_frob}  and Theorem~\ref{theo_generic_integral}. In order to state these results, we recall the definition of \emph{radius of convergence of a differential operator with coefficients in $\mathcal{W}_{p}\cap\mathbb{C}_p[[z]]$}. For a real number $r>0$ we have the following ring of analytic functions $$\mathcal{A}(z,r):=\left\{\sum_{j\geq0}a_j(x-z)^j\in\mathcal{W}_{p}[[x-z]]:\text{ for all  } s<r\text{, } \lim\limits_{j\rightarrow\infty}|a_j|_{\mathcal{G}}s^j=0\right\}.$$
In other words, $\mathcal{A}(z,r)$ is the ring of power series with coefficients in $\mathcal{W}_p$ that converge in the open disk $D(z,r):=\{x\in\mathcal{W}_p: |x-z|<r\}$.  

Now, let us consider $\bm{\tau}:\mathcal{W}_{p}\cap\mathbb{C}_p[[z]]\rightarrow\mathcal{A}(z,1)$ given by $$\bm{\tau}(f)=\sum_{j\geq0}\frac{(d/dz)^j(f)}{j!}(x-z)^j.$$

The map $\bm{\tau}$ is well-defined because, for all $f\in\mathcal{W}_{p}\cap\mathbb{C}_p[[z]]$, $\left|\frac{(d/dz)^j(f)}{j!}\right|_{\mathcal{G}}\leq|f|_{\mathcal{G}}$ for all $j\geq0$. It is clear that $\bm{\tau}$ is a homomorphism of rings. 

\begin{rema}\label{rema_unit}\hfill

\begin{enumerate}
\item According to Proposition 1.2 of~\cite{Gillesalgebriques}, a nonzero element $f=\sum_{n\in\mathbb{Z}}a_nz^n\in\mathcal{W}_p$ is a unit of $\mathcal{W}_p$ if and only if there is $n_0\in\mathbb{Z}$ such that $|f|_{\mathcal{G}}=|a_{n_0}|$.  In particular, every nonzero element of $\mathbb{Z}_p[[z]]$ is a unit of $\mathcal{W}_p$.\\

\item Let $f$ be in $\mathcal{W}_{p}\cap\mathbb{C}_p[[z]]$. If $f$ is a unit of $\mathcal{W}_{p}$ then $\bm{\tau}(f)$ is a unit of $\mathcal{A}(z,1)$. Indeed, let us write $$\bm{\tau}(f)=f(1+g),\text{ where } g=\sum_{j\geq1}\frac{(d/dz)^j(f)}{j!f}(x-z)^j.$$
We have $g\in\mathcal{A}(z,1)$ because, for all $j\geq1$, $\left|\frac{(d/dz)^j(f)}{j!}\right|_{\mathcal{G}}\leq|f|_{\mathcal{G}}$ and, by assumption, $1/f\in\mathcal{W}_{p}.$ Further, $1+g$ is a unit element of $\mathcal{A}(z,1)$ given that $$(1+g)\left(\sum_{k\geq0}(-1)^kg^k\right)=1\text{ and }\sum_{k\geq0}(-1)^kg^k\in\mathcal{A}(z,1).$$ Thus, $$\frac{1}{f}\sum_{k\geq0}(-1)^kg^k\in\mathcal{A}(z,1).$$ Finally, it is clear that $$\bm{\tau}(f)\left(\frac{1}{f}\sum_{k\geq0}(-1)^kg^k\right)=1.$$
\end{enumerate}
\end{rema}

For all $f\in\mathcal{W}_{p}\cap\mathbb{C}_p[[z]]$, we have
\begin{equation}\label{eq_com_der}
\bm{\tau}(d/dz f)=d/dx(\bm{\tau} f).
\end{equation}
The ring $(\mathcal{W}_{p}\cap\mathbb{C}_p[[z]])[[x-z]]$ is equipped with the endomorphism $$\bm{F}\left(\sum_{j\geq0}a_j(z)(x-z)^j\right)=\sum_{j\geq0}a_j(z^p)(x^p-z^p)^j.$$
 Since
 \begin{equation*}
 x^p-z^p=\bm{\tau}(z^p)-z^p=\sum_{i=1}^{p}\frac{(d/dz)^i(z^p)}{i!}(x-z)^i,
 \end{equation*}
 we have
 \begin{equation}\label{eq_com_frob}
 \bm{\tau}\circ\bm{F}=\bm{F}\circ\bm{\tau}.
 \end{equation}
 Let $L$ be a monic differential operator of order $n$ with coefficients in $\mathcal{W}_{p}\cap\mathbb{C}_p[[z]]$ and let $A$ be the companion matrix of $L$. Then the differential system $\delta_x X=\bm{\tau}(A)X$ ($\delta_x=xd/dx$) has a unique solution $\mathcal{U}\in GL_n( \mathcal{W}_{p}[[x-z]])$ such that $\mathcal{U}(z)=I$. Moreover, $$\mathcal{U}=\sum_{j\geq0}\frac{A_j}{j!z^j}(x-z)^j,$$
where $A_0=I$ and $A_{j+1}=\delta A_{j}+A_j(A-jI)$ for $j\geq0$. Following \cite[Chap. III, p. 94]{Dworkgfunciones}, the \emph{radius of convergence of $L$} is the radius of convergence of the matrix $\mathcal{U}$ in $\mathcal{W}_{p}$. We let  $\bm{r}(L)$ denote the radius of convergence of $L$. So, $\mathcal{U}\in M_n(\mathcal{A}(z,\bm{r}(L)))$. For a matrix $C=(c_{i,j})_{1\leq i,j\leq n}$ with coefficients in $\mathcal{W}_{p}$, we set $||C||=\max\{|c_{i,j}|_{\mathcal{G}}\}_{1\leq i,j\leq n}$. So, $$\frac{1}{\bm{r}(L)}=\limsup_{j\rightarrow\infty}\left|\left|\frac{A_j}{j!}\right|\right|^{1/j}.$$

 \begin{rema}\label{rema_a_j}
 Let $L$ be a differential operator with coefficients in $\mathbb{Z}_p[[z]]$. If $\bm{r}(L)\geq1$ then $\lim\limits_{j\rightarrow\infty}||A_j||=0$.   Indeed, as $\bm{r}(L)\geq1$ then $\limsup\limits_{j\rightarrow\infty}||A_j/j!||^{1/j}\leq1$. Therefore, $\lim\limits_{j\rightarrow\infty}||A_j||=0$.
 \end{rema}

 \begin{prop}\label{prop_p_integral_frob} 
 Let $L$ be a differential operator with coefficients in $\mathbb{Z}_p[[z]]$ having a $p$\nobreakdash-integral Frobenius structure. Then $\bm{r}(L)\geq1$.
 \end{prop}
  
 \begin{proof}
 Let $A$ be the companion matrix of $L$. We put $A_0=I$ and $A_{j+1}=\delta A_{j}+A_j(A-jI)$ for $j\geq0$. We want to see that $\bm{r}(L)\geq1$. Since $L$ belongs to $\mathbb{Z}_p[[z]][\delta]$, we have $||A||=1$ and thus, for all $j\geq0$, $||A_j||\leq 1$. So, $\bm{r}(L)\geq |p|^{1/p-1}$. We know that $\delta_x(\mathcal{U})=\bm{\tau}(A)\mathcal{U}$ and thus, $$\bm{F}(\mathcal{U})=\sum_{j\geq0}\frac{A_j(z^{p})}{j!z^{jp}}(x^{p}-z^{p})^j$$ is a solution of the system  $\delta_xX=p\bm{F}(\bm{\tau}(A))X$. 
 
 We now prove that  $\bm{F}(\mathcal{U})\in M_n(\mathcal{A}(z,\bm{r}(L)^{1/p}))$. As the Gauss norm is non\nobreakdash-Archimedean, that is equivalent to showing that if $|x_0-z|_{\mathcal{G}}<\bm{r}(L)^{1/p}$ then $\lim\limits_{j\rightarrow\infty}\left|\left|\frac{A_j(z^{p})}{j!z^{jp}}\right|\right||x_0^{p}-z^{p}|_{\mathcal{G}}^j=0$. Indeed, if $|x_0-z|_{\mathcal{G}}<\bm{r}(L)^{1/p}$ then $|x_0-z|^{p}_{\mathcal{G}}<\bm{r}(L)$ but $(x_0-z)^{p}=x_0^{p}-z^{p}+\sum_{i=1}^{p-1}(-1)^i\binom{p}{i}x_0^{p-i}z^{i}$ and $$\left|\sum_{i=1}^{p-1}(-1)^i\binom{p}{i}x_0^{p-i}z^{i}\right|_{\mathcal{G}}\leq1/p$$ because the Gauss norm is non\nobreakdash-Archimedean and, by Lucas' Theorem, $|\binom{p}{i}|\leq 1/p$ for all $1\leq i<p$. Furthermore, $1/p<|p|^{1/p-1}\leq\bm{r}(L)$. Thus, we get $$|x_{0}^{p}-z^{p}|_{\mathcal{G}}\leq\max\{|x_0-z|_{\mathcal{G}}^{p},1/p\}<\bm{r}(L).$$ Given that the radius of convergence of $\mathcal{U}$ is $\bm{r}(L)$, we have $\lim\limits_{j\rightarrow\infty}\left|\left|\frac{A_j}{j!z^j}\right|\right||x_0^{p}-z^{p}|_{\mathcal{G}}^j=0$. As $\left|\left|\frac{A_j}{j!z^j}\right|\right|=\left|\left|\frac{A_j(z^{p})}{j!z^{jp}}\right|\right|$, we obtain $\lim\limits_{j\rightarrow\infty}\left|\left|\frac{A_j(z^{p})}{j!z^{jp}}\right|\right||x_0^{p}-z^{p}|_{\mathcal{G}}^j=0$. Thus $\bm{F}(\mathcal{U})\in M_n(\mathcal{A}(z,\bm{r}(L)^{1/p}))$, where $n$ is the order of $L$. Now, by assumption, there is $\Phi\in M_n(\mathbb{Z}_p[[z]])$ such that $\det(\Phi)\neq0$ and $\delta\Phi=A\Phi-p\Phi A(z^{p})$. Hence, by Equations~\eqref{eq_com_der} and \eqref{eq_com_frob}, we obtain $$\delta_x(\bm{\tau}(\Phi))=\bm{\tau}(A)\bm\tau(\Phi)-p\bm{\tau}(\Phi)\bm{F}(\bm{\tau}(A)).$$
So, $\bm{\tau}(\Phi)\bm{F}(\mathcal{U})$ is a solution of the system $\delta_xX=\bm{\tau}(A)X$. Since $\det(\Phi)\neq0$ and $\det(\Phi)\in\mathbb{Z}_p[[z]]$, it follows from (1) of Remark~\ref{rema_unit} that $\det(\Phi)$ is a unit of $\mathcal{W}_{p}$. So, by (2) of Remark~\ref{rema_unit}, we conclude that $\bm{\tau}(\Phi)\in GL_n(\mathcal{A}(z,1))$. Consequently, $\bm{\tau}(\Phi)\bm{F}(\mathcal{U})\in M_n(\mathcal{A}(z,r))$, where $r=\min\{1,\bm{r}(L)^{1/p}\}$. Given that $\mathcal{U}\in GL_n(\mathcal{W}_{p}[[x-z]])$ is a solution of the system $\delta_X=\bm{\tau}(A)X$, there is $C\in GL_n(\mathcal{W}_{p})$ such that $\bm{\tau}(\Phi)\bm{F}(\mathcal{U})=\mathcal{U}C$. As the radius of convergence of $\mathcal{U}C$ is still $\bm{r}(L)$ then $\bm{r}(L)\geq r$. Consequently, $\bm{r}(L)\geq1$.
 \end{proof}
 
It follows from Remark~\ref{rem_sol} that for a MUM differential operator $L\in\mathbb{Q}_p[[z]][\delta]$ of order $n\geq2$ there are unique power series $\mathfrak{f}(z)\in 1+z\mathbb{Q}_p[[z]]$ and $\mathfrak{g}(z)\in z\mathbb{Q}_p[[z]]$ such that $y_0=\mathfrak{f}$ and $y_1=\mathfrak{f}\log z+\mathfrak{g}$ are solutions of $L$. The main ingredient to prove (1) of Theorem~\ref{theo_integral_mirror} is the following result.

 
 \begin{theo}\label{theo_generic_integral}
 Let $L\in\mathbb{Z}_p[[z]][\delta]$ be a MUM differential operator. If $\bm{r}(L)\geq1$ then $y_0(z)\in 1+z\mathbb{Z}_p[[z]]$.
  \end{theo}

\section{$p$-adic differential equations}\label{sec_p_diff_equa}

The crucial ingredient in the proof of Theorem~\ref{theo_generic_integral} is Proposition~\ref{prop_ant} which is recursively obtained from Lemma~\ref{lemm_paso_base}. The proof of this lemma relies on the theory of $p$\nobreakdash-adic differential equations.

We recall that the \emph{Cartier operator} $\Lambda_p:\mathbb{C}_p[[z]]\rightarrow\mathbb{C}_p[[z]]$ is given by $\Lambda_p(\sum_{i\geq0}a(i)z^i)=\sum_{i\geq0}a(ip)z^i$. From the definition it is clear that, if $f(z)$ and $g(z)$ belong to $\mathbb{C}_p[[z]]$ then $\Lambda_p(f(z)g(z^p))=\Lambda_p(f)g(z)$. Given a matrix $B\in M_n(\mathbb{C}_p[[z]])$, $\Lambda_p(B)$ is the matrix obtained by applying $\Lambda_p$ to each entry of $B$.
 
 \begin{lemm}\label{lemm_paso_base}
 Let $L$ be a MUM differential operator of order $n$ in $\mathbb{Z}_p[[z]][\delta]$, let $A$ be the companion matrix of $L$, and let $Y_L$ be the uniform part of $\delta X=AX$ and let $N=A(0)$. If $\bm{r}(L)\geq1$ then there exists a MUM differential operator $L_1$ of order $n$ in $\mathbb{Z}_p[[z]][\delta]$ such that:
 \begin{enumerate}[label=(\alph*')]
   \item  the fundamental matrix of solutions of $\delta X=B_1X$  is given by \[\mathrm{diag}(1,1/p,\ldots, 1/p^{n-1})\Lambda_p(Y_L)\mathrm{diag}(1,p,\ldots, p^{n-1})z^{N},\] 
 where $B_1$ is the companion matrix of $L_1$,
\item the matrix \[H_1=Y_L(\Lambda_p(Y_L)(z^{p}))^{-1}\mathrm{diag}(1,p,p^2,\ldots, p^{n-1})\] belongs to $M_n(\mathbb{Z}_p[[z]])\cap GL_n(\mathcal{W}_{\mathbb{Q}_p}\cap\mathbb{Q}_p[[z]])$,  $||H_1||=1$, and $$\delta(H_1)=AH_1-p{H}_1B_1(z^{p}),$$
\item $\bm{r}(L_1)\geq1$.
 \end{enumerate}
 \end{lemm}
 
 In order to prove Lemma~\ref{lemm_paso_base},  we first prove Lemma~\ref{lemm_paso_0} whose proof is based upon the ideas found in \cite{christolpadique}.

\begin{lemm}\label{lemm_paso_0}
 Let $L$ be a MUM differential operator of order $n$ in $\mathbb{Z}_p[[z]][\delta]$, let $A$ be the companion matrix of $L$, and let $Y_L$ be the uniform part of $\delta X=AX$ and let $N=A(0)$. If $\bm{r}(L)\geq1$ then:
 \begin{enumerate}
 \item the matrix $H_0=(\Lambda_p(Y_L)(z^p))Y_L^{-1}(z)$ belongs to $GL_n(\mathbb{Z}_p[[z]])$,
 \item there is $F\in M_n(\mathcal{W}_{\mathbb{Q}_p}\cap\mathbb{Q}_p[[z]])$ such that $$\delta H_0=pF(z^p)H_0-H_0A,$$
 \item  the fundamental matrix of solutions of $\delta X=FX$ is given by $\Lambda_p(Y_L)z^{\frac{1}{p}N}$.

 \end{enumerate}

 \end{lemm}
 
\begin{proof}
 (1). Let us consider the sequence $\{A_j(z)\}_{j\geq0}$, where $A_0(z)=I$ is the identity matrix and $A_{j+1}(z)=\delta A_j(z)+A_j(z)(A(z)-jI)$. Given that $A\in M_n(\mathbb{Z}_p[[z]])$ then $A_j(z)\in M_n(\mathbb{Z}_p[[z]])$  for all $j\geq0$.  Since $X^p-1=(X-1)((X-1)^{p-1}+pt(X))$ with $t(X)\in\mathbb{Z}[X]$, it follows that if $\xi^p=1$ then $|\xi-1|\leq|p|^{1/(p-1)}$. Furthermore, for all integers $j\geq1$, we have $|p|^{j/(p-1)}<|j!|$. Thus, for all $j\geq1$,  $|(\xi-1)^j/j!|<1$. Since the norm is non\nobreakdash-Archimedean and $\sum_{\xi^p=1}(\xi-1)^j/j$ is a rational number, we conclude that, for all $j\geq0$, $\left|\sum_{\xi^p=1}\frac{(\xi-1)^j}{j!}\right|\leq1/p$. Thus, $\left|\sum_{\xi^p=1}\frac{(\xi-1)^j}{pj!}\right|\leq1$. That is, for all $j\geq0$, $\sum_{\xi^p=1}\frac{(\xi-1)^j}{pj!}\in\mathbb{Z}_p$. So, for all $j\geq0$, the matrix $$A_j\left(\sum_{\xi^p=1}\frac{(\xi-1)^j}{pj!}\right)$$
 belongs to $M_n(\mathbb{Z}_p[[z]])$. We set $$H_0=\sum_{j\geq0}A_j(z)\left(\sum_{\xi^p=1}\frac{(\xi-1)^j}{pj!}\right).$$
 Let us show that $H_0\in M_n(\mathbb{Z}_p[[z]])$. That is equivalent to saying that 
 \begin{equation}\label{eq_lim}
 \lim\limits_{j\rightarrow\infty}||A_j||\left|\sum_{\xi^p=1}\frac{(\xi-1)^j}{pj!}\right|=0
 \end{equation}
  because $\mathbb{Z}_p[[z]]$ is complete with respect to the Gauss norm and it is non\nobreakdash-Archimedean. Indeed, as $\bm{r}(L)\geq1$ then, by Remark~\ref{rema_a_j}, $\lim\limits_{j\rightarrow\infty}||A_j||=0$. Therefore, \eqref{eq_lim} follows immediately since, for all $j\geq0$, $\left|\sum_{\xi^p=1}\frac{(\xi-1)^j}{pj!}\right|\leq1$. 
   
 Now, we are going to prove that $H_0=\Lambda_p(Y_L)(z^p)Y_L^{-1}$. Let us write $Y_L=\sum_{j\geq0}Y_jz^j$. We have $N^n=0$ because $L$ is MUM and thus, it  follows from \cite[p.165]{christolpadique} that $H_0Y_L=\sum_{j\geq0}Y_{jp}z^{jp}$. So, $H_0Y_L=\Lambda_p(Y_L)(z^p)$. Therefore, $H_0=\Lambda_p(Y_L)(z^p)Y_L^{-1}$. Finally, $H_0\in GL_n(\mathbb{Z}_p[[z]])$ because $H_0(0)$ is the identity matrix.
 
(2). Now, we set  
 \begin{equation}\label{eq_frob_m}
 F(z)=[\delta(\Lambda_p(Y_L))+\frac{1}{p}\Lambda_p(Y_L)N](\Lambda_p(Y_L))^{-1}.
 \end{equation}
As $H_0Y_L=\Lambda_p(Y_L)(z^{p})$ then, from Equation \eqref{eq_frob_m} we obtain, $$pF(z^{p})=[p(\delta(\Lambda_p(Y_L)))(z^{p})+H_0Y_LN][Y_L^{-1}H_0^{-1}].$$ We also have $$\delta(H_0)Y_L+H_0(\delta Y_L)=\delta(H_0Y_L)=\delta(\Lambda_p(Y_L)(z^{p}))=p(\delta(\Lambda_p(Y_L)))(z^{p}).$$  Since $\delta(Y_Lz^{N})=AY_Lz^{N}$, it follows that $\delta Y_L=AY_L-Y_LN$. Thus, \[p(\delta(\Lambda_p(Y_L)))(z^{p})=\delta(H_0)Y_L+H_0[AY_L-Y_LN].\]
 Then, 
 \begin{align*}
 pF(z^{p})&=[p(\delta(\Lambda_p(Y_L)))(z^{p})+H_0Y_LN][Y_L^{-1}H_0^{-1}]\\
 &=[(\delta H_0)Y_L+H_0AY_L][Y_L^{-1}H_0^{-1}]\\
 &=(\delta H_0)H_0^{-1}+H_0AH_0^{-1}.
 \end{align*}
 Consequently,  $$\delta H_0=pF(z^{p})H_0-H_0A.$$ Since $pF(z^p)=(\delta H_0+H_0A)H_{0}^{-1}\in M_n(\mathbb{Z}_p[[z]])$ we obtain $F(z)\in M_n(\mathcal{W}_{\mathbb{Q}_p}\cap\mathbb{Q}_p[[z]])$. 
 
 (3.) Finally, we show that the matrix $\Lambda_p(Y_L)X^{\frac{1}{p}N}$ is the fundamental matrix of solutions of $\delta X=FX$. It is clear that $\Lambda_p(Y_L)(0)=I$ and from Equality~\eqref{eq_frob_m}, we obtain $F(z)\Lambda_p(Y_L)=\delta(\Lambda_p(Y_L))+\Lambda_p(Y_L)\frac{1}{p}N$. Thus,
\begin{align*}
\delta(\Lambda_p(Y_L)X^{\frac{1}{p}N})&=\delta(\Lambda_p(Y_L))X^{\frac{1}{p}N}+\Lambda_p(Y_L)\frac{1}{p}NX^{\frac{1}{p}N}\\
&=(\delta(\Lambda_p(Y_L))+\Lambda_p(Y_L)\frac{1}{p}N)X^{\frac{1}{p}N}\\
&=F(z)\Lambda_p(Y_L)X^{\frac{1}{p}N}.
\end{align*}
This completes the proof.
\end{proof}

The following result is known as "Dwork-Frobenius" Theorem and it is usually proven for differential operators with coefficients in $E_{p}$ (see \cite[Proposition 8.1]{C81}). For the sake of completeness we prove it following the same lines of \cite[Proposition 8.1]{C81}.

 \begin{prop}\label{prop_DF}
Let $L=\delta^n+a_{n-1}(z)\delta^{n-1}+\cdots+a_1(z)\delta+a_0(z)$ be a differential operator with $a_i(z)\in\mathcal{W}_{p}\cap\mathbb{C}_p[[z]]$ for all $0\leq i< n$. If $\bm{r}(L)\geq1$ then $|a_i|_{\mathcal{G}}\leq 1$ for all $0\leq i< n$.
\end{prop}

\begin{proof}
For every $r<1$, the ring $\mathcal{A}(z,1)$ is equipped with the following absolute value $|\sum_{n\geq0}f_n(x-z)^n|_r=\sup\{|f_n|_{\mathcal{G}}r^n\}_{n\geq0}$. This absolute value extends in a natural way to $\mathcal{M}(z,1):=\mathrm{Frac}(\mathcal{A}(z,1))$ and, for all $g\in\mathcal{M}(z,1)$, $|\delta_x g/g|_r\leq1$. Let us consider the differential operator $\mathcal{L}=\delta^n_x+\bm{\tau}(a_{n-1})\delta_x^{n-1}+\cdots+\bm{\tau}(a_1)\delta_x+\bm{\tau}(a_0)$. By assumption, there are $f_1,\ldots, f_n\in\mathcal{A}(z,1)$ linearly independent over $\mathrm{Frac}(\mathcal{W}_{p})$ such that $\mathcal{L}(f_i)=0$ for every $1\leq i\leq n$. By induction, we set $g_1=f_1$ and $g_i=\mathcal{L}_{i-1}\circ\cdots\circ\mathcal{L}_1(f_i)$ with $\mathcal{L}_i=\delta_x-\delta_xg_i/g_i$. So, $g_1,\ldots, g_n$ belong to $\mathcal{M}(z,1)$ and $f_1,\ldots, f_n$ are solutions of $\mathcal{L}_n\circ\cdots\circ\mathcal{L}_1$. Since $\mathcal{L}_n\circ\cdots\circ\mathcal{L}_1=\delta_x^{n}+\cdots$, the differential operator $\mathcal{L}-\mathcal{L}_n\circ\cdots\circ\mathcal{L}_1$ has order at least $n-1$ and has $f_1,\ldots, f_n$ as solutions. Thus,  $\mathcal{L}=\mathcal{L}_n\circ\cdots\circ\mathcal{L}_1$. As, for every $1\leq i\leq n$, $|\delta_x g_i/g_i|_r\leq 1$ then, for every $0\leq i< n$, $|\bm{\tau}(a_i)|_{r}\leq1$.  Since $a_i\in\mathcal{W}_{p}\cap\mathbb{C}_p[[z]]$, $\left|\frac{(d/dz)^j(a_i)}{j!}\right|_{\mathcal{G}}\leq|a_i|_{\mathcal{G}}$ for all $j\geq0$. Consequently, $|\bm{\tau}(a_i)|_r=|a_i|_{\mathcal{G}}$. Therefore, $|a_i|_{\mathcal{G}}\leq 1$ for all $0\leq i< n$.
\end{proof}

We are now ready to prove Lemma~\ref{lemm_paso_base}.
  
  \begin{proof}[Proof of Lemma~\ref{lemm_paso_base}]
 According to Lemma~\ref{lemm_paso_0}, we have $F\in M_n(\mathcal{W}_{\mathbb{Q}_p}\cap\mathbb{Q}_p[[z]])$ such that 
 \begin{equation}\label{eq_equiv2}
 \delta H_0=pF(z^p)H_0-H_0A,
 \end{equation}
with $H_0=\Lambda_p(Y_L)(z^p)Y_L^{-1}\in GL_n(\mathbb{Z}_p[[z]])$. Let us write $F=(b_{i,j})_{1\leq i,j\leq n}$. We set $$L_1:=\delta^n-b_{n,n}\delta^{n-1}-\frac{1}{p}b_{n,n-1}\delta^{n-2}-\cdots\frac{1}{p^{n-i}}b_{n,i}\delta^{i-1}-\cdots-\frac{1}{p^{n-2}}b_{n,2}\delta-\frac{1}{p^{n-1}}b_{n,1}.$$

It is clear that the coefficients of $L_1$ belong to $\mathcal{W}_{\mathbb{Q}_p}\cap\mathbb{Q}_p[[z]]$. Actually, we are going to see at the end of this proof that $L_1\in\mathbb{Z}_p[[z]][\delta]$. From Equality~\eqref{eq_equiv2}, we deduce that $pF(0)=N$. Since $L$ is MUM, the last row of $N$ is equal to zero. So, for all $i\in\{1,\ldots, n-1\}$, $b_{n,i}(0)=0$. For this reason $L_1$ is MUM.

(a') Let $B_1$ be the companion matrix of $L_1$. We are going to see that  \[T=\mathrm{diag}(1,1/p,\ldots, 1/p^{n-1})\Lambda_p(Y_L)\mathrm{diag}(1,p,\ldots, p^{n-1})z^{N},\] is the fundamental matrix of solutions of $\delta X=B_1X$. For this purpose, we split the proof into three steps.

\textbf{First step} Let us write $Y_L=(f_{i,j})_{1\leq i,j\leq n}$. Then, for all integers $i\in\{1,\ldots, n-1\}$ and $k\in\{1,\ldots,n\}$, $f_{i,k-1}+\delta f_{i,k}=f_{i+1,k}$. In fact, for every $i,j\in\{1,\ldots,n\}$, we set $$\phi_{i,j}=\sum_{k=1}^jf_{i,k}\frac{(\log z)^{j-k}}{(j-k)!}.$$
Notice that $Y_Lz^{N}=(\phi_{i,j})_{1\leq i,j\leq n}$. Since $Y_Lz^{N}$ is the fundamental matrix of $\delta X=AX$ and $A$ is the companion matrix of $L$, it follows that, for every $i\in\{1,\ldots,n-1\}$ and $j\in\{1,\ldots, n\}$, $\phi_{i+1,j}=\delta\phi_{i,j}.$  Therefore,  $$\sum_{k=1}^jf_{i+1,k}\frac{(\log z)^{j-k}}{(j-k)!}=\phi_{i+1,j}=\delta\phi_{i,j}=\delta(f_{i,1})\frac{(\log z)^{j-1}}{(j-1)!}+\sum_{k=2}^j(f_{i,k-1}+\delta(f_{i,k}))\frac{(\log z)^{j-k}}{(j-k)!}.$$
As $\log z$ is transcendental over $\mathbb{Q}_p[[z]]$, it follows from the previous equality that, for all $k\in\{1,\ldots, n\}$, $f_{i,k-1}+\delta f_{i,k}=f_{i+1,k}$. 

\textbf{Second step} For every $j\in\{1,\ldots,n\}$, we set $$\theta_j=\sum_{k=1}^{j}p^{k-1}\Lambda_p(f_{1,k})\frac{(\log z)^{j-k}}{(j-k)!}.$$
We are going to show that $\theta_j$ is a solution of $L_1$ for every $1\leq j\leq n$. 

Let us write $\Lambda_p(Y_L)X^{\frac{1}{p}N}=(\eta_{i,j})_{1\leq i,j\leq n}$. Then, for every $i,j\in\{1,\ldots,n\}$, $$\eta_{i,j}=\sum_{k=1}^j\Lambda_p(f_{i,k})\frac{(\log z)^{j-k}}{p^{j-k}(j-k)!}.$$
On the one hand, we are going to see that, for every $j,l\in\{1,\ldots,n\}$, $$\frac{1}{p^{n-l}}\delta^{l-1}\theta_j=\frac{1}{p^{n-j}}\eta_{l,j}.$$ To prove this equality, we are going to proceed by induction on $l\in\{1,\ldots, n\}$. For $l=1$, it is clear that $\frac{1}{p^{n-1}}\theta_j=\frac{1}{p^{n-j}}\eta_{1,j}$. Now, we suppose that for some $l\in\{1,\ldots,n-1\}$, we have $\frac{1}{p^{n-l}}\delta^{l-1}\theta_j=\frac{1}{p^{n-j}}\eta_{l,j}$. So, $\frac{1}{p^{n-l}}\delta^{l}\theta_j=\frac{1}{p^{n-j}}\delta(\eta_{l,j})$. We prove now that $\delta(\eta_{l,j})=\frac{1}{p}\eta_{l+1,j}$. In fact, as  $\Lambda_p\circ\delta=p\delta\circ\Lambda_p$ and, according to the first step, $f_{l,k-1}+\delta(f_{l,k})=f_{l+1,k}$ for all $k\in\{1,\ldots, n\}$ then
\begin{align*}
\delta(\eta_{l,j})&=\sum_{k=1}^j\delta(\Lambda_p(f_{l,k}))\frac{(\log z)^{j-k}}{p^{j-k}(j-k)!}+\Lambda_p(f_{l,k})\frac{(\log z)^{j-k-1}}{p^{j-k}(j-k-1)!}\\
&=\sum_{k=1}^j\frac{1}{p}\Lambda_p(\delta(f_{l,k}))\frac{(\log z)^{j-k}}{p^{j-k}(j-k)!}+\frac{1}{p}\Lambda_p(f_{l,k})\frac{(\log z)^{j-k-1}}{p^{j-k-1}(j-k-1)!}\\
&=\frac{1}{p}\left[\Lambda_p(\delta(f_{l,1}))\frac{(\log z)^{j-1}}{p^{j-1}(j-1)!}+\sum_{k=2}^j\Lambda_p(f_{l,k-1}+\delta f_{l,k}))\frac{(\log z)^{j-k}}{p^{j-k}(j-k)!}\right]\\
&=\frac{1}{p}\left[\Lambda_p(f_{l+1,1})\frac{(\log z)^{j-1}}{p^{j-1}(j-1)!}+\sum_{k=2}^j\Lambda_p(f_{l+1,k})\frac{(\log z)^{j-k}}{p^{j-k}(j-k)!}\right]\\
&=\frac{1}{p}\eta_{l+1,j}.
\end{align*}
Thus, we obtain  $$\frac{1}{p^{n-l}}\delta^{l}\theta_j=\frac{1}{p^{n-j}}\delta(\eta_{l,j})=\frac{1}{p^{n-j}}(\frac{1}{p}\eta_{l+1,j}).$$ Whence, $$\frac{1}{p^{n-l-1}}\delta^{l}\theta_j=\frac{1}{p^{n-j}}\eta_{l+1,j}.$$ So, for every $j,l\in\{1,\ldots,n\}$, $\frac{1}{p^{n-l}}\delta^{l-1}\theta_j=\frac{1}{p^{n-j}}\eta_{l,j}$.

On the other hand, by Lemma~\ref{lemm_paso_0}, we know that $\Lambda_p(Y_L)X^{\frac{1}{p}N}$ is the fundamental matrix of $\delta X=FX$. Hence, for every $j\in\{1,\ldots,n\}$, \[F\begin{pmatrix} 
\eta_{1,j}\\
\eta_{2,j}\\
\vdots\\
\eta_{n,j}\\
\end{pmatrix}
=\begin{pmatrix} 
\delta(\eta_{1,j})\\
\delta(\eta_{2,j})\\
\vdots\\
\delta(\eta_{n,j}))\\
\end{pmatrix}.
\]
In particular, for every $j\in\{1,\ldots,n\}$, $$b_{n,1}\eta_{1,j}+b_{n,2}\eta_{2,j}+\cdots+b_{n,k}\eta_{k,j}+\cdots+b_{n,n}\eta_{n,j}=\delta(\eta_{n,j}).$$
Multiplying the previous equality by $1/p^{n-j}$ and by using the fact that, for every $l\in\{1,\ldots, n\}$, $\frac{1}{p^{n-j}}\eta_{l,j}=\frac{1}{p^{n-l}}\delta^{l-1}\theta_j$, we get  $$\frac{1}{p^{n-1}}b_{n,1}\theta_{j}+\frac{1}{p^{n-2}}b_{n,2}\delta\theta_{j}+\cdots+\frac{1}{p^{n-k}}b_{n,k}\delta^{k-1}\theta_{j}+\cdots+b_{n,n}\delta^{n-1}\theta_{j}=\delta^n\theta_j.$$
Therefore, $\theta_j$ is a solution of $L_1$. 

\textbf{Third step} We are going to see that $T=(\delta^{i-1}\theta_j)_{1\leq i,j\leq n}$.  As we have already seen, for every $i,j\in\{1,\ldots,n\}$, $$\frac{1}{p^{n-i}}\delta^{i-1}\theta_j=\frac{1}{p^{n-j}}\eta_{i,j}.$$ Hence,  for every $i,j\in\{1,\ldots,n\}$,
\begin{align*}
\delta^{i-1}\theta_j=\frac{p^j}{p^i}\eta_{i,j}&=\frac{p^j}{p^i}\sum_{k=1}^j\Lambda_p(f_{i,k})\frac{(\log z)^{j-k}}{p^{j-k}(j-k)!}=\sum_{k=1}^j\frac{\Lambda_p(f_{i,k})}{p^{i-k}}\frac{(\log z)^{j-k}}{(j-k)!}.
\end{align*}
But \[\mathrm{diag}(1,1/p,\ldots, 1/p^{n-1})\Lambda_p(Y_L)\mathrm{diag}(1,p,\ldots,p^{n-1})=\left(\frac{\Lambda_p(f_{i,k})}{p^{i-k}}\right)_{1\leq i,k\leq n}.
\]
By definition $$T=\mathrm{diag}(1,1/p,\ldots, 1/p^{n-1})\Lambda_p(Y_L)\mathrm{diag}(1,p,\ldots, p^{n-1})z^{N}.$$ Thus, it follows that $$T=(\delta^{i-1}\theta_j)_{1\leq i,j\leq n}.$$ 

Finally, notice that $\theta_1,\ldots,\theta_n$ are linearly independent over $\mathbb{Q}_p$ because $\log z$ is transcendental over $\mathbb{Q}_p[[z]]$. Since $B_1$ is the companion matrix of $L_1$ and, according to the second step, $\theta_1,\ldots, \theta_n$ are solutions of $L_1$, it follows that the matrix $(\delta^{i-1}\theta_j)_{1\leq i,j\leq n}$ is a fundamental matrix of solutions of $\delta X=B_1X$. Since $T=(\delta^{i-1}\theta_j)_{1\leq i,j\leq n}$, we conclude that $T$ is the fundamental matrix of $\delta X=B_1X$.

(b') 
From Lemma~\ref{lemm_paso_0}, we know that the matrix $H_0^{-1}=Y_L(\Lambda_p(Y_L)(z^p))^{-1}$ belongs to $GL_n(\mathbb{Z}_p[[z]])$. Thus,  $$H_1=Y_L(\Lambda_p(Y_L)(z^{p}))^{-1}\mathrm{diag}(1,p,\ldots, p^{n-1})\in GL_n(\mathcal{W}_{\mathbb{Q}_p}\cap\mathbb{Q}_p[[z]])\cap M_n(\mathbb{Z}_p[[z]]).$$ Since $H^{-1}_0\in M_n(\mathbb{Z}_p[[z]])$ and $H^{-1}_0(0)$ is the identity matrix, we have $||H_0^{-1}||=1$. Thereby, $||H_1||\leq 1$. But $H_1(0)=\mathrm{diag}(1,p,\ldots, p^{n-1})$ and thus, $||H_1||=1$. Now, we are going to show that $$\delta(H_1)=AH_1-pH_1B_1(z^{p}).$$
We infer from (a') that
 \[T(z^p)=\mathrm{diag}(1,1/p,\ldots, 1/p^{n-1})\Lambda_p(Y_L)(z^p)\mathrm{diag}(1,p,\ldots, p^{n-1})X^{pN}\] 
is the fundamental matrix of solutions of $\delta X=pB_1(z^p)X$. Furthermore, \[H_1T(z^p)=Y_L\mathrm{diag}(1,p,\ldots, p^{n-1})X^{pN}.\] 
But \[\mathrm{diag}(1,p,\ldots, p^{n-1})X^{pN}=z^{N}\mathrm{diag}(1,p,\ldots, p^{n-1}). 
\] 
Consequently, \[H_1T(z^p)=Y_Lz^{N}\mathrm{diag}(1,p,\ldots, p^{n-1}).\]
Thus, $H_1T(z^p)$ is a fundamental matrix of $\delta X=AX$. So that, we have 
\begin{align*}
AH_1T(z^p)&=\delta(H_1T(z^p))\\
&=\delta(H_1)T(z^p)+H_1\delta(T(z^p))\\
&=\delta(H_1)T(z^p)+H_1(pB_1(z^p))T(z^p).
\end{align*}
Hence, 
\begin{equation}\label{eq_equiv_lambd_frob}
\delta(H_1)=AH_1-pH_1B_1(z^{p}).
\end{equation}
This completes the proof of (b').

(c') From Equations~\eqref{eq_com_der}, \eqref{eq_com_frob}, and ~\eqref{eq_equiv_lambd_frob} we get 
\begin{equation}\label{eq_tau_equiv}
\delta_x(\bm{\tau}(H_1))=\bm{\tau}(A)\bm{\tau}(H_1)-p\bm{\tau}(H_1)\bm{F}(\bm{\tau}(B_1)).
\end{equation}
We set $C_0=I$ and $C_{j+1}=\delta C_j+C_j(B_1-jI)$ for $j\geq0$. We want to see that $\bm{r}(L_1)\geq1$. By definition, that is equivalent to proving that for all $r<1$, $\lim\limits_{j\rightarrow\infty}\left|\left|\frac{C_j}{j!}\right|\right|r^j=0$. We know that $$\mathcal{\widetilde{U}}=\sum_{j\geq0}\frac{C_j(z)}{j!z^{j}}(x-z)^j\in GL_n(\mathcal{W}_p[[x-z]])$$ is a solution of the system $\delta_xX=\bm{\tau}(B_1)X$. Therefore, $\delta_x\bm{F}(\mathcal{\widetilde{U}})=p\bm{F}(\bm{\tau}(B_1))\bm{F}(\mathcal{\widetilde{U}})$. By Equation~\eqref{eq_tau_equiv}, we get that $\bm{\tau}(H_1)\bm{F}(\mathcal{\widetilde{U}})$ is a solution of the system $\delta_xX=\bm{\tau}(A)X$. Since $H_1\in GL_n(\mathcal{W}_{\mathbb{Q}_p}\cap\mathbb{Q}_p[[z]])$, we have $\bm{\tau}(H_1)\in GL_n(\mathcal{A}(z,1))\cap GL_n((\mathcal{W}_{p}[[x-z]])$. As $\bm{F}(\mathcal{\widetilde{U}})\in GL_n(\mathcal{W}_{p}[[x-z]])$ then $\bm{\tau}(H_1)\bm{F}(\mathcal{\widetilde{U}})\in GL_n(\mathcal{W}_{p}[[x-z]])$ and since $\bm{\tau}(H_1)\bm{F}(\mathcal{\widetilde{U}})$ is a solution of the system $\delta_xX=\bm{\tau}(A)X$, there is $C\in GL_n(\mathcal{W}_{p})$ such that $\bm{\tau}(H_1)\bm{F}(\mathcal{\widetilde{U}})=\mathcal{U}C$. Thus, $\bm{F}(\mathcal{\widetilde{U}})=\mathcal{U}C\bm{\tau}(H_1)^{-1}$. By assumption, $\mathcal{U}$ belongs to $M_n(\mathcal{A}(z,1))$ and we know that $\bm{\tau}(H_1)$ belongs to $GL_n(\mathcal{A}(z,1))$. Whence, $\bm{F}(\mathcal{\widetilde{U}})\in M_n(\mathcal{A}(z,1))$. By definition, $$\bm{F}(\mathcal{\widetilde{U}})=\sum_{j\geq0}\frac{C_j(z^p)}{j!z^{jp}}(x^p-z^p)^j.$$
Thus, for all $r<1$, $\lim\limits_{j\rightarrow\infty}\left|\left|\frac{C_j(z^p)}{j!z^{jp}}\right|\right|r^j=0$. So, for all $r<1$, $\lim\limits_{j\rightarrow\infty}\left|\left|\frac{C_j}{j!}\right|\right|r^j=0$ since $\left|\left|\frac{C_j(z^p)}{j!z^{jp}}\right|\right|=\left|\left|\frac{C_j}{j!z^{j}}\right|\right|$. Consequently, $\bm{r}(L_1)\geq1$.

Finally, we are in a position to apply Proposition~\ref{prop_DF} to $L_1$ and we conclude that the coefficients of $L_1$ have norm less than or equal to 1. But we know that the coefficients of $L_1$ belong to $\mathbb{Q}_p[[z]]$ and so $L_1\in\mathbb{Z}_p[[z]][\delta]$.
 \end{proof}
 
 \begin{prop}\label{prop_ant}
Let $L$ be a MUM differential operator of order $n$ in $\mathbb{Z}_p[[z]][\delta]$, let $A$ be the companion matrix of $L$, $N=A(0)$, and let $Y_Lz^{N}$ be the fundamental matrix of solutions of $\delta X=AX$. If $\bm{r}(L)\geq1$ then, for every integer $m>0$, there exists a MUM differential operator $L_m$ of order $n$ in $\mathbb{Z}_p[[z]][\delta]$ such that:
 \begin{enumerate}[label=(\alph*)]
 
  \item  the fundamental matrix of solutions of $\delta X=B_mX$  is given by \[
\mathrm{diag}(1,1/p^m,\ldots, 1/p^{m(n-1)})\Lambda^m_p(Y_L)\mathrm{diag}(1,p^m,\ldots, p^{m(n-1)})z^{N},\] 
where $B_m$ is the companion matrix of $L_m$,

\item let $H_m=Y_L(\Lambda^m_p(Y_L)(z^{p^m}))^{-1}\mathrm{diag}(1,p^m,p^{2m}\ldots, p^{m(n-1)})$. Then $$H_m\in M_n(\mathbb{Z}_p[[z]])\cap GL_n(\mathcal{W}_{\mathbb{Q}_p}\cap\mathbb{Q}_p[[z]]),$$ $||H_m||=1$, and $$\delta(H_m)=AH_m-p^mH_mB_m(z^{p^m}),$$

\item $\bm{r}(L_m)\geq1$.
 \end{enumerate}
 \end{prop}

\begin{proof}
We proceed by induction on $m\in\mathbb{Z}_{>0}$. For $m=1$ we are in a position to apply Lemma~\ref{lemm_paso_base} and thus, there is a MUM differential operator $L_1$ of order $n$ in $\mathbb{Z}_p[[z]][\delta]$ satisfying the conditions (a)-(c). Now, suppose that for some integer $m>0$ there is a MUM differential operator $L_m$ of order $n$ in $\mathbb{Z}_p[[z]][\delta]$ satisfying the conditions (a)-(c), that is, if $B_m$ is the companion matrix of $L_m$ then we have

\textbullet\quad the fundamental matrix of solutions of the system $\delta X=B_mX$ is given by $Y_{L_m}z^{N}$, where 
\begin{equation}\label{eq_l_m}
Y_{L_m}=\mathrm{diag}(1,1/p^m,\ldots, 1/p^{m(n-1)})\Lambda^m_p(Y_L)\mathrm{diag}(1,p^m,\ldots, p^{m(n-1)}),
\end{equation}
\textbullet\quad the matrix $H_m$ given by  \[Y_L(\Lambda^m_p(Y_L)(z^{p^m}))^{-1}\mathrm{diag}(1,p^m,\ldots, p^{m(n-1)})
\]
belongs to  $M_n(\mathbb{Z}_p[[z]])\cap GL_n(\mathcal{W}_{\mathbb{Q}_p}\cap\mathbb{Q}_p[[z]])$, $||H_m||=1$, and
\begin{equation}\label{eq_equiv}
\delta(H_m)=AH_m-p^mH_mB_m(z^{p^m}),
\end{equation}

\textbullet\quad $\bm{r}(L_m)\geq1$.

 So, we are in a position to apply Lemma~\ref{lemm_paso_base} to the differential operator $L_{m}$ and thus, there exists a MUM differential operator $L_{m+1}$ of order $n$  in $\mathbb{Z}_p[[z]][\delta]$ such that \[\mathrm{diag}(1,1/p,\ldots, 1/p^{(n-1)})\Lambda_p(Y_{L_m})\mathrm{diag}(1,p,\ldots, p^{(n-1)})z^{N}\] 
 is the fundamental matrix of solutions of $\delta X=B_{m+1}X$, where $B_{m+1}$ is the companion matrix of $L_{m+1}$. So, by Equation~\eqref{eq_l_m}, the fundamental matrix of solutions of  $\delta X=B_{m+1}X$ can be written as follows \[\mathrm{diag}(1,1/p^{m+1},\ldots, 1/p^{(m+1)(n-1)})\Lambda^{m+1}_p(Y_L)\mathrm{diag}(1,p^{m+1},\ldots, p^{(m+1)(n-1)})z^{N}.\]
 
  By invoking Lemma~\ref{lemm_paso_base} again, we know that the matrix \[H_1=Y_{L_m}(\Lambda_p(Y_{L_m})(z^{p}))^{-1}\mathrm{diag}(1,p,\ldots, p^{n-1})\
\]
belongs to $M_n(\mathbb{Z}_p[[z]])\cap GL_n(\mathcal{W}_{\mathbb{Q}_p}\cap\mathbb{Q}_p[[z]])$, $||H_1||=1$, and $\delta(H_1)=B_mH_1-pH_1B_{m+1}(z^p)$. Thus, 
\begin{equation}\label{eq_equiv_nivel_m}
\delta(H_1(z^{p^m}))=p^mB_m(z^{p^m})H_1(z^{p^m})-H_1(z^{p^m})p^{m+1}B_{m+1}(z^{p^{m+1}}).
\end{equation}
We put $H_{m+1}=H_{m}(z)H_1(z^{p^m})$. Then $H_{m+1}\in M_n(\mathbb{Z}_p[[z]])\cap GL_n(\mathcal{W}_{\mathbb{Q}_p}\cap\mathbb{Q}_p[[z]])$ and \[H_{m+1}=Y_L(z)((\Lambda^{m+1}_pY_L)(z^{p^{m+1}}))^{-1}\mathrm{diag}(1,p^{m+1},\ldots, p^{(m+1)(n-1)}).\]
Furthermore, from Equations~\eqref{eq_equiv} and \eqref{eq_equiv_nivel_m}, we obtain $$\delta(H_{m+1})=AH_{m+1}-p^{m+1}H_{m+1}B_{m+1}(z^{p^{m+1}}).$$
Finally, $||H_{m+1}||=1$ given that $||H_{1}||=1=||H_{m}||$ and \[H_{m+1}(0)=\mathrm{diag}(1,p^{m+1},\ldots, p^{(m+1)(n-1)}).\]
Therefore, the conditions (a) and (b) also hold true for $m+1$. Finally, by (c') of Proposition~\ref{lemm_paso_base}, we obtain $\bm{r}(L_{m+1})\geq1$. For this reason the condition (c) also holds true for $m+1$.  
\end{proof}

\section{Proof of Theorem \ref{theo_generic_integral}}\label{sec_p_integral}

Let us write $y_0(z)=\sum_{j\geq0}f_jz^j$. In order to prove that $y_0(z)\in 1+z\mathbb{Z}_p[[z]]$, it is sufficient to show that, for all integers $m>0$, $f_0, f_1,\ldots, f_{p^m-1}\in\mathbb{Z}_p$. Let $m>0$ be an integer and let $A$ be the companion matrix of $L$. By hypotheses, we know that $\bm{r}(L)\geq1$. Therefore, by (b) of Proposition~\ref{prop_ant}, there exist $B_m\in M_n(\mathbb{Z}_p[[z]])$ and $H_m\in M_n(\mathbb{Z}_{p}[[z]])\cap GL_n(\mathcal{W}_{\mathbb{Q}_p}\cap\mathbb{Q}_p[[z]])$ such that  $||H_m||=1$ and 
\begin{equation}\label{eq_equiv_m}
\delta(H_m)=AH_m-p^mH_mB_m(z^{p^m}).
\end{equation}
We set $\widetilde{y}_0=\Lambda_p^{m}(y_0(z))$. Thanks to (a) of Proposition~\ref{prop_ant}, the vector $(\widetilde{y}_0,\delta\widetilde{y}_0,\ldots, \delta^{n-1}\widetilde{y}_0)$ is solution of the of the system $\delta\vec{y}=B_m\vec{y}$ Hence, it follows from Equation~\eqref{eq_equiv_m} that
 \[H_m\begin{pmatrix}
\widetilde{y}_0(z^{p^m}) \\
(\delta\widetilde{y}_0)(z^{p^m})  \\
\vdots \\
(\delta^{n-1}\widetilde{y}_0)(z^{p^m})  \\
\end{pmatrix}
\]

 is a solution of the system $\delta\vec{y}=A\vec{y}$.  If we put $H_m=(h_{i,j}(z))_{1\leq i,j\leq n}$ then $$h_{1,1}(z)\widetilde{y}_0(z^{p^m})+h_{1,2}(z)(\delta\widetilde{y}_0)(z^{p^m})+\cdots+h_{1,n}(z)(\delta^{n-1}\widetilde{y}_0)(z^{p^m})$$
is a solution of $L$ because $A$ is the companion matrix of $L$. Furthermore, this solution belongs to $\mathbb{Q}_p[[z]]$ given that $H_m\in M_n(\mathbb{Z}_p[[z]])$ and $y_0(z)\in\mathbb{Q}_p[[z]]$. Thus, according to Remark~\ref{rem_sol}, there exists $c\in\mathbb{Q}_p$ such that 
\begin{equation}\label{eq_eigenvalue}
h_{1,1}(z)\widetilde{y}_0(z^{p^m})+h_{1,2}(z)(\delta\widetilde{y}_0)(z^{p^m})+\cdots+h_{1,n}(z)(\delta^{n-1}\widetilde{y}_0)(z^{p^m})=cy_0(z).
\end{equation}
As $y_0(0)=1$ and, for all integers $j>0$,  $(\delta^{j}\widetilde{y}_0)(z^{p^m})(0)=0$ then $h_{1,1}(0)=c$. But, by (b) of Proposition~\ref{prop_ant}, we conclude that $h_{1,1}(0)=1$. So that, $c=1$. By using (b) of Proposition~\ref{prop_ant} again, we infer that, for all $j\in\{2,\ldots, n\}$, $h_{1,j}(0)=0$ and it is clear that, for all integers $j>0$, $(\delta^{j}\widetilde{y}_0)(z^{p^m})\in z^{p^{m}}\mathbb{Q}_p[[z]]$. Hence, from Equation~\eqref{eq_eigenvalue}, we obtain 
\begin{equation}\label{eq_reduction}
h_{1,1}(z)(\Lambda^m_p(y_0))(z^{p^m})\equiv y_0(z)\bmod z^{p^{m}+1}\mathbb{Q}_p[[z]].
\end{equation}
Let us write $h_{1,1}(z)=\sum_{j\geq0}h_jz^j$. Thus, by Equation~\eqref{eq_reduction}, we conclude that, for every integer $k\in\{1,\ldots, p^{m}-1\}$, $h_k=f_k$. But $h_{1,1}(z)\in\mathbb{Z}_p[[z]]$ since $H_m\in M_n(\mathbb{Z}_p[[z]])$. Consequently, for every integer $k\in\{1,\ldots, p^{m}-1\}$, $f_k\in\mathbb{Z}_p.$
$\hfill\square$

We end this section by proving the following result.

\begin{coro}\label{coro_picard_fuchs}
Let $L$ be an irreducible Picard-Fuchs equation with coefficients in $\mathbb{Q}(z)$ and $y_0(z)\in 1+z\mathbb{Q}[[z]]$ solution of $L$.  If $L$ is MUM type then, for almost every prime number $p$, $y_0(z)\in 1+z\mathbb{Z}_p[[z]]$.
\end{coro}
In order to prove this corollary, we need the following result.
\begin{prop}\label{prop_frob_p_integral}
Let $L$ be an irreducible MUM differential operator with coefficients in $\mathbb{Q}(z)$. If $L$ has a Frobenius structure for almost every prime number $p$ then $L$ has a $p$\nobreakdash-integral Frobenius structure, for almost every prime number $p$. Moreover, if $n$ is the order of $L$ then, for almost every prime number $p$, there is $\Phi_p=(\phi_{i,j})_{1\leq i,j\leq n}\in M_n(\mathbb{Z}_p[[z]])\cap GL_n(E_p)$ such that $\Phi_p$ is a $p$\nobreakdash-integral Frobenius matrix for $L$ and $||\Phi_p||=1$. 
\end{prop}

In order to prove this proposition, we recall that Ax-Sen-Tate's Theorem  ensures that if $\xi\in\mathbb{C}_p$ and $\sigma(\xi)=\xi$ for any $\sigma\in \mathrm{Gal}(\mathbb{C}_p/\mathbb{Q}_p)$ then $\xi\in\mathbb{Q}_p$.

\begin{proof}[Proof of Proposition~\ref{prop_frob_p_integral}]
Let $\mathcal{P}$ be the set of prime numbers. By assumption, there is a set $\mathcal{S}_1$ of prime numbers such that $\mathcal{P}\setminus\mathcal{S}_1$ is finite and, for all $p\in\mathcal{S}_1$, $L$ has a Frobenius structure for $p$. Thus, according to  Propositions~4.1.2, 4.6.4, and 4.7.2 of \cite{C83}, for all $p\in\mathcal{S}_1$, we have
 
(a) the radius of convergence of $L$ is greater than or equal to 1.

 Thus, following Proposition 5.1 and Theorem 6.1 of ~\cite[Chap. III]{Dworkgfunciones}, the singularities of $L$ are all regular and the exponents of $L$ are rational numbers. Thus, there is a set $\mathcal{S}_2$ of prime numbers such that $\mathcal{P}\setminus\mathcal{S}_2$ is finite and

(b) if $\alpha$, $\beta$ are two different singularities of $L$ then $|\alpha-\beta|_p=1$ for all $p\in\mathcal{S}_2$,

(c) if $\gamma$ is an exponent of $L$ then $|\gamma|_p=1$ for all $p\in\mathcal{S}_2$.

We set $\mathcal{S}=\mathcal{S}_1\cap\mathcal{S}_2$. Then $\mathcal{P}\setminus\mathcal{S}$ is finite. We are going to see that, for all $p\in\mathcal{S}$, $L$ has a $p$\nobreakdash-integral Frobenius structure. Let $p$ be in $\mathcal{S}$. Thus, there exists $G_p\in GL_n(E_{p})$ such that $\delta G_p=AG_p-pG_pA(z^{p})$. Let us write $G_p=(g_{i,j})_{1\leq i,j\leq n}$ and suppose that $||G_p||=|g_{k,l}|_{\mathcal{G}}$. Since $g_{k,l}$ is a nonzero element of the field $E_{p}$ and $E_p\subset\mathcal{W}_p$, we get that $g_{k,l}$ is a unit element of $\mathcal{W}_p$. So, by (1) of Remark~\ref{rema_unit},  there is $c\in\mathbb{C}_p$ such that $c$ is a coefficient of $g_{k,l}$ and $|g_{k,l}|_{\mathcal{G}}=|c|$. Note that $c$ is not zero because $g_{k,l}$ is not zero. We put $\Phi_p=\frac{1}{c}G_p$. So $\delta\Phi_p=A\Phi_p-p\Phi_pA(z^{p})$, $||\Phi_p||=1$, and 1 is a coefficient of  $\frac{1}{c}g_{k,l}$. Furthermore, it is clear that $\Phi_p\in GL_n(E_p)$. Let $\sigma$ be in $\mathrm{Gal}(\mathbb{C}_p/\mathbb{Q}_p)$. So, $\sigma$ naturally extends to $\mathcal{W}_{p}$ as endomorphism of rings, given by $\sigma(\sum_{n\in\mathbb{Z}}a_nz^n)=\sum_{n\in\mathbb{Z}}\sigma(a_n)z^n$. Further, $\sigma(E_{p})=E_{p}$ and $\delta\circ\sigma=\sigma\circ\delta$. For a matrix $T\in M_n(E_{p})$, the matrix $T^{\sigma}$ is the matrix obtained after applying $\sigma$ to each entry of $T$. Since $\delta\Phi_p=A\Phi_p-p\Phi_pA(z^{p})$ and $A\in M_n(\mathbb{Q}(z))$, we have $\delta(\Phi_p^{\sigma})=A\Phi_p^{\sigma}-p\Phi_p^{\sigma}A(z^{p})$.  As $L$ is irreducible and the conditions (a), (b), and (c) are satisfied because $p\in\mathcal{S}$, we can apply \cite[Lemma]{Dwork89} and therefore,  $\Phi_p^{\sigma}=d\Phi_p$ for some $d\in\mathbb{C}_p$. But $1$ is a coefficient of $\frac{1}{c}g_{k,l}$ and thus $1=\sigma(1)=d$. So, $\Phi_p^{\sigma}=\Phi_p$. Since $\sigma$ is an arbitrary element of $\mathrm{Gal}(\mathbb{C}_p/\mathbb{Q}_p)$, by Ax-Sen-Tate's Theorem, we conclude that $\Phi_p$ is a matrix with coefficients in $\mathcal{W}_{\mathbb{Q}_p}$. But we know that $||\Phi_p||=1$ and consequently, $\Phi_p\in M_n(\mathcal{W}_{\mathbb{Z}_p})$. We also have $\det(\Phi_p)\neq0$ because $\det(G_p)\neq0$ and $c\neq0$.

Now, we are going to see that $\Phi_p\in M_n(\mathbb{Z}_p[[z]])$. Let $Y_Lz^{N}$ be the fundamental matrix of solutions of $\delta X= AX$, where $A$ is the companion matrix of $L$ and $N=A(0)$. Then, by Remark~\ref{rem_det}, we conclude that $\Phi_p=Y_LCY_L(z^{p})^{-1}\in M_n(\mathbb{C}_p[[z]])$, where $C\in M_n(\mathbb{C}_p)$. So, $\Phi_p\in M_n(\mathbb{C}_p[[z]])\cap M_n(\mathcal{W}_{\mathbb{Z}_p})$. Therefore, $\Phi_p\in M_n(\mathbb{Z}_p[[z]])$ and we have already seen that $\Phi_p\in GL_n(E_p)$.
\end{proof}

\begin{proof}[Proof of corollary~\ref{coro_picard_fuchs}]
Since $L$ is a Picard-Fuchs equations, it follows from  \cite[Theorem 22.2.1]{Kedlaya} that $L$ has a strong Frobenius structure for almost every prime number $p$. Thus, according to Proposition~\ref{prop_frob_p_integral}, for almost every prime number $p$, $L$  has a $p$-integral Frobenius structure. Then, by Proposition~\ref{prop_p_integral_frob} and Theorem~\ref{theo_generic_integral}, we conclude that, for almost every primer number $p$, $y_0(z)\in 1+z\mathbb{Z}_p[[z]]$.
\end{proof}

\section{Some $p$-integrality properties}\label{sec_proof_mirror_maps}
In order to prove (1) and (2) of Theorem~\ref{theo_integral_mirror}, we need some additional results that deal mainly with the $p$\nobreakdash-integrality of some power series with coefficients in $\mathbb{Q}_p$. In particular, in Theorem~\ref{prop_arith} we prove that $\exp(y_1/y_0)^p\in z\mathbb{Z}_p[[z]]$.

\begin{lemm}\label{prop_p_int}
Let $\mathfrak{A}:=a_0+\frac{a_1}{p}\delta+\cdots+\frac{a_{n-1}}{p^{n-1}}\delta^{n-1}$ with $a_0(z),\ldots, a_{n-1}(z)\in\mathbb{Z}_p[[z]]$ and let $u\in 1+z\mathbb{Q}_p[[z]]$. If $\mathfrak{A}(u(z^p))=u$ then $u\in 1+z\mathbb{Z}_p[[z]]$.
\end{lemm}

\begin{proof}
We set $\bm\omega_0=1$ and $\bm\omega_{i+1}=\mathfrak{A}(\bm\omega_i(z^p))$ for $i\geq0$. Thus, for all integers $m\geq1$, $\bm\omega_m\in 1+z\mathbb{Z}_p[[z]]$ because, for all integers $j\geq$1, $\delta^j(\bm\omega_i(z^p))=p^j(\delta^j\bm\omega_i)(z^p)$.
 We are going to see that, for all integers $m\geq0$, $\bm\omega_m\equiv u\bmod z^{p^{m}}\mathbb{Z}_p[[z]]$. As $u(0)=1$ then $\bm\omega_0\equiv u\bmod z\mathbb{Z}_p[[z]]$. Now, we suppose that $\bm\omega_m\equiv u\bmod z^{p^{m}}\mathbb{Z}_p[[z]]$ for some integer $m\geq0$. Then, $\bm\omega_m(z^p)\equiv u(z^p)\bmod z^{p^{m+1}}\mathbb{Z}_p[[z]]$.
Further, we have $$\bm\omega_{m+1}-u=\mathfrak{A}(\bm\omega_m(z^{p}))-\mathfrak{A}(u(z^p))=\sum_{i=0}^{n-1}a_i\left[(\delta^i\bm\omega_m)(z^p)-(\delta^iu)(z^p)\right].$$
Since $\bm\omega_m(z^p)-u(z^p)\in z^{p^{m+1}}\mathbb{Z}_p[[z]]$, for all integers $i\geq0$, we have \[\delta^i(\bm\omega_m(z^p)-u(z^p))\in p^{(m+1)i}z^{p^{m+1}}\mathbb{Z}_p[[z]].\] Nevertheless, $\delta^i(\bm\omega_m(z^p))=p^i(\delta^i(\bm\omega_m))(z^p)$ and similarly, $\delta^i(u(z^p))=p^i(\delta^i(u))(z^p)$. Thus, $$p^i[(\delta^i\bm\omega_m)(z^p)-(\delta^iu)(z^p)]=\delta^i[\bm\omega_m(z^p)-u(z^p)]\in p^{(m+1)i}z^{p^{m+1}}\mathbb{Z}_p[[z]].$$
Whence, $$(\delta^i\bm\omega_m)(z^p)-(\delta^iu)(z^p)\in p^{mi}z^{p^{m+1}}\mathbb{Z}_p[[z]]\subset z^{p^{m+1}}\mathbb{Z}_p[[z]].$$
Therefore, $\bm\omega_{m+1}-u\in z^{p^{m+1}}\mathbb{Z}_p[[z]]$. Then, we conclude that, for all integers $m\geq0$, $\bm\omega_m\equiv u\bmod z^{p^{m}}\mathbb{Z}_p[[z]]$. Consequently, $u\in 1+z\mathbb{Z}_p[[z]]$ because  for all integers $m\geq1$, $\bm\omega_m\in 1+z\mathbb{Z}_p[[z]]$.
\end{proof}
 
 \begin{lemm}\label{prop_eig}
 Let $L$ and $D$ be  MUM differential operators in $\mathbb{Q}_p[[z]][\delta]$ of order $n$ and let $A$ and $B$ be the respective companion matrices of $L$ and $D$. Let $y_0=\mathfrak{f}(z)$, $y_1=\mathfrak{f}(z)\log z+\mathfrak{g}(z)$ be solutions of $L$ and let $\widetilde{y}_0=\mathfrak{f}_1(z)$, $\widetilde{y}_1=\mathfrak{f}_1(z)\log z+\mathfrak{g}_1(z)$ be solutions of $D$, with $\mathfrak{f}(z), \mathfrak{f}_1(z)\in1+z\mathbb{Q}_p[[z]]$ and $\mathfrak{g}(z),\mathfrak{g}_1(z)\in z\mathbb{Q}_p[[z]]$. Suppose that there is $H=(h_{i,j}(z))_{1\leq i,j\leq n}\in M_n(\mathbb{Q}_p[[z]])$ such that $\delta H=AH-pHB(z^p)$. Then, 
$$\widetilde{\omega}(z^p)\alpha_0y_0+(\delta\widetilde{\omega})(z^p)r_2+\cdots+(\delta^{n-1}\widetilde{\omega})(z^p)r_{n}y_0=y_0(p\alpha_0\omega+\alpha_1),$$
where $\widetilde{\omega}=\widetilde{y}_1/\widetilde{y}_0$, $\omega={y}_1/{y}_0$, $\alpha_0=h_{1,1}(0)$, $\alpha_1=h_{1,2}(0)$, and for $2\leq j\leq n$, $r_j=\sum_{i=j}^n\binom{i-1}{i-j}h_{1,i}(z)(\delta^{i-j}\widetilde{y}_0)(z^p)$.
 \end{lemm}
 
 \begin{proof}
 Since $\widetilde{y}_0$ is a solution of $D$, it follows from the equality $\delta H=AH-pHB(z^p)$ that $\sum_{i=1}^nh_{1,i}(z)(\delta^{i-1}\widetilde{y}_0)(z^p)$ is a solution of $L$. It is clear that  $\sum_{i=1}^nh_{1,i}(z)(\delta^{i-1}\widetilde{y}_0)(z^p)\in\mathbb{Q}_p[[z]]$. So, according to Remark~\ref{rem_sol}, there is $c\in\mathbb{Q}_p$ such that $\sum_{i=1}^nh_{1,i}(z)(\delta^{i-1}\widetilde{y}_0)(z^p)=c\mathfrak{f}(z)$. Thereby, $c=h_{1,1}(0)=\alpha_0$. In a similar way, since $\widetilde{y}_1$ is a solution of $D$, it follows from the equality $\delta H=AH-pHB(z^p)$ that $\sum_{i=1}^nh_{1,i}(z)(\delta^{i-1}\widetilde{y}_1)(z^p)$ is a solution of $L$. As $\sum_{i=1}^nh_{1,i}(z)(\delta^{i-1}\widetilde{y}_1)(z^p)\in \mathbb{Q}_p[[z]]+\mathbb{Q}_p[[z]]\log z$, it follows from Remark~\ref{rem_sol} that there are $c_0, c_1\in\mathbb{Q}_p$ such that $\sum_{i=1}^nh_{1,i}(z)(\delta^{i-1}\widetilde{y}_1)(z^p)=c_0y_0+c_1y_1$. Now, for all integers $m\geq1$, $$\delta^{m}\widetilde{y}_1=\delta^{m}(\mathfrak{f}_1(z))\log z+m\delta^{m-1}\mathfrak{f}_1(z)+\delta^m\mathfrak{g}_1(z).$$
So,
\begin{align*}
\sum_{i=1}^nh_{1,i}(z)(\delta^{i-1}\widetilde{y}_1)(z^p)&=\left[\sum_{i=1}^nph_{1,i}(z)(\delta^{i-1}\mathfrak{f}_1)(z^p)\right]\log z+\sum_{i=0}^{n-1}h_{1,i+1}(z)(i\delta^{i-1}\mathfrak{f}_1+\delta^{i}\mathfrak{g}_1)(z^p)\\
&=c_0y_0+c_1y_1\\
&=c_1\mathfrak{f}(z)\log z+c_0\mathfrak{f}(z)+c_1\mathfrak{g}(z).
\end{align*}
Thus,  $$\sum_{i=1}^nph_{1,i}(z)(\delta^{i-1}\mathfrak{f}_1)(z^p)=c_1\mathfrak{f}(z)\quad\text{ and }\quad \sum_{i=0}^{n-1}h_{1,i+1}(z)(i\delta^{i-1}\mathfrak{f}_1+\delta^{i}\mathfrak{g}_1)(z^p)=c_0\mathfrak{f}(z)+c_1\mathfrak{g}(z).$$ From the right equality, we have $c_0=h_{1,2}(0)$ because $\mathfrak{g}(0)=0=\mathfrak{g}_1(0)$ and $\mathfrak{f}(0)=1=\mathfrak{f}_1(0)$. So, $c_0=\alpha_1$. Further, we have already seen that $\sum_{i=1}^nh_{1,i}(z)(\delta^{i-1}\mathfrak{f}_1)(z^p)=\alpha_0\mathfrak{f}$. Thus $c_1=p\alpha_0$. 

We put $\widetilde{\omega}=\frac{\widetilde{y}_1}{\widetilde{y}_0}$. Then $\widetilde{\omega}\widetilde{y}_0=\widetilde{y}_1$ and $$\sum_{i=1}^nh_{1,i}(z)(\delta^{i-1}(\widetilde{\omega}\widetilde{y}_0))(z^p)=\alpha_1y_0+p\alpha_0y_1.$$

 Given that, for all integers $m\geq1$, $\delta^m(\widetilde{\omega}\widetilde{y}_0)=\sum_{i=0}^m\binom{m}{i}\delta^{m-i}\widetilde{\omega}\cdot\delta^{i}\widetilde{y}_0$, we obtain 
 \begin{align*}
&\sum_{i=1}^nh_{1,i}(z)(\delta^{i-1}(\widetilde{\omega}\widetilde{y}_0))(z^p)=\widetilde{\omega}(z^p)\left[\sum_{i=1}^nh_{1,i}(z)(\delta^{i-1}\widetilde{y}_0)(z^p)\right]+\\
&(\delta\widetilde{\omega})(z^p)\left[\sum_{i=2}^n\binom{i-1}{i-2}h_{1,i}(z)(\delta^{i-1}\widetilde{y}_0)(z^p)\right]+(\delta^2\widetilde{\omega})(z^p)\left[\sum_{i=3}^n\binom{i-1}{i-3}h_{1,i}(z)(\delta^{i-2}\widetilde{y}_0)(z^p)\right]+\\
&\cdots+(\delta^{n-1}\widetilde{\omega})(z^p)h_{1,n}(z)(\widetilde{y}_0)(z^p).
 \end{align*}
 We know that $$\sum_{i=1}^nh_{1,i}(z)(\delta^{i-1}\widetilde{y}_0)(z^p)=\alpha_0y_0$$ and, for every $j\in\{2,\ldots,n\}$, we set $r_j=\sum_{i=j}^n\binom{i-1}{i-j}h_{1,i}(z)(\delta^{i-j}\widetilde{y}_0)(z^p).$ So, 
 \begin{align*}
 \sum_{i=1}^nh_{1,i}(\delta^{i-1}(\widetilde{\omega}\widetilde{y}_0))(z^p)=\widetilde{\omega}(z^p)\alpha_0y_0+(\delta\widetilde{\omega})(z^p)r_2+\cdots+(\delta^{n-1}\widetilde{\omega})(z^p)r_{n}.
 \end{align*}
Consequently,
 $$\widetilde{\omega}(z^p)\alpha_0y_0+(\delta\widetilde{\omega})(z^p)r_2+\cdots+(\delta^{n-1}\widetilde{\omega})(z^p)r_{n}=\alpha_1y_0+p\alpha_0y_1=y_0(p\alpha_0\omega+\alpha_1).$$
 This finishes the proof. \end{proof}
 
Theorem~\ref{theo_generic_integral} and Lemma~\ref{lemm_paso_base} are essential in the proof of the following two propositions.
 
 \begin{prop}\label{prop_lambda}
 Let $L$ be a MUM differential operator of order $n$ with coefficients in $\mathbb{Z}_p[[z]]$. Let $y_0=\mathfrak{f}(z)$ and  $y_1=\mathfrak{f}(z)\log z+\mathfrak{g}$ be solutions of $L$ such that $\mathfrak{f}(z)\in 1+z\mathbb{Q}_p[z]]$ and $\mathfrak{g}(z)\in z\mathbb{Q}_p[[z]]$. Suppose that $L$ has a $p$\nobreakdash-integral Frobenius structure given by the matrix $\Phi=(\phi_{i,j}(z))_{1\leq i,j\leq n}$. If $|\phi_{1,1}(0)|=1$ then $\delta\widetilde{\omega}\in 1+z\mathbb{Z}_p[[z]]$, where $\widetilde{\omega}=\widetilde{y_1}/\widetilde{y_0}$ with $\widetilde{y}_0=\Lambda_p(\mathfrak{f}(z))$ and $\widetilde{y}_1=\Lambda_p(\mathfrak{f}(z))\log z+p\Lambda_p(\mathfrak{g}(z))$.
 \end{prop}
 
 \begin{proof}
Let $Y_Lz^{N}$ be the fundamental matrix of solutions of $\delta X=AX$, where $A$ is the companion matrix of $L$. Since $L$ has a $p$\nobreakdash-integral Frobenius structure,  according to Proposition~\ref{prop_p_integral_frob}, $\bm{r}(L)\geq1$. As  $L\in\mathbb{Z}_p[[z]][\delta]$ is MUM and $\bm{r}(L)\geq1$ we can apply  Lemma~\ref{lemm_paso_base} and thus, from (a') of Lemma~\ref{lemm_paso_base}, there is a MUM differential operator $L_1\in\mathbb{Z}_p[[z]][\delta]$ of order $n$ such that the fundamental matrix of solutions of $\delta X=B_1X$ is given by  $$\mathrm{diag}(1,1/p,\ldots,1/p^{n-1})\Lambda_p(Y_L)\mathrm{diag}(1,p,\ldots, p^{n-1})z^{N},$$
where $B_1$ is the companion matrix of $L_1$. Therefore, $$\widetilde{y}_0:=\Lambda_p(\mathfrak{f}(z))\quad\text{ and }\quad\widetilde{y}_1:=\Lambda_p(\mathfrak{f}(z))\log z+p\Lambda_p(\mathfrak{g}(z))$$
 are solutions of $L_1$. Let us take $$T=\mathrm{diag}(1,1/p,\ldots,1/p^{n-1})\Lambda_p(\Phi)Y_L(z)(\Lambda_p(Y_L)(z^{p}))^{-1}\mathrm{diag}(1,p,\ldots,p^{n-1}).$$ We are going to see that  $\delta T=B_1T-pTB_1(z^p)$. In fact, from (b') of Lemma~\ref{lemm_paso_base}, we deduce that $\delta(H_1^{-1})=pB_1(z^p)H_1^{-1}-H_1^{-1}A$, where $$H_1=Y_L(z)(\Lambda_p(Y_L)(z^{p}))^{-1}\mathrm{diag}(1,p,p^2,\ldots, p^{n-1}).$$ By hypotheses, we know that $\delta\Phi=A\Phi-p\Phi A(z^p)$. Thus, $$\delta(H_1^{-1}\Phi)=pB_1(z^p)H_1^{-1}\Phi-pH_1^{-1}\Phi A(z^p).$$ As $\Lambda_p\circ\delta=p\delta\circ\Lambda_p$ then, we get $$\delta(\Lambda_p(H_1^{-1}\Phi))=B_1\Lambda_p(H_1^{-1}\Phi)-\Lambda_p(H_1^{-1}\Phi)A.$$
By invoking (b') of Lemma~\ref{lemm_paso_base}  again , we have $\delta(H_1)=AH_1-pH_1B_1(z^p)$. Thus, $$\delta(\Lambda_p(H_1^{-1}\Phi)H_1)=B_1\Lambda_p(H_1^{-1}\Phi)H_1-p\Lambda_p(H_1^{-1}\Phi)H_1B_1(z^p).$$
From Remark~\ref{rem_det}, we know that $\Phi=Y_L\Phi(0)Y_L(z^p)^{-1}$. Whence,  $$\Lambda_p(\Phi)=\Lambda_p(Y_L)\Phi(0)Y_L(z)^{-1}$$ and consequently, $$\Lambda_p(H_1^{-1}\Phi)H_1=T.$$ Let us write $T=(t_{i,j}(z))_{1\leq i,j\leq n}$. We have $T\in M_n(\mathbb{Q}_p[[z]])$ because $Y_L\in GL_n(\mathbb{Q}_p[[z]])$ and, by assumption, $\Phi\in M_n(\mathbb{Z}_p[[z]])$. As  $\delta T=B_1T-pTB_1(z^p)$ then, by Lemma~\ref{prop_eig}, we deduce that  $$\widetilde{\omega}(z^p)\alpha_0\widetilde{y}_0+(\delta\widetilde{\omega})(z^p)r_2+\cdots+(\delta^{n-1}\widetilde{\omega})(z^p)r_{n}=\widetilde{y}_0(p\alpha_0\widetilde{\omega}+\alpha_1),$$
where $\alpha_0=t_{1,1}(0)$, $\alpha_1=t_{1,2}(0)$ and $r_j=\sum_{i=j}^n\binom{i-1}{i-j}t_{1,i}(z)(\delta^{i-j}\widetilde{y}_0)(z^p)$.

Since $|\phi_{1,1}(0)|=1$ and  $||\Phi||=1$, we have $||\Lambda_p(\Phi)||=1$. In addition, by Lemma~\ref{lemm_paso_0}, we know that $Y_L(z)(\Lambda_p(Y_L)(z^{p}))^{-1}\in GL_n(\mathbb{Z}_p[[z]])$. Therefore, $t_{1,2},\ldots, t_{1,n}\in p\mathbb{Z}_p[[z]]$ and consequently, $r_2,\ldots, r_{n}\in p\mathbb{Z}_p[[z]]$ and $\alpha_1=p\beta$ with $\beta\in\mathbb{Z}_p$. So, $r_j=ps_j$ with $s_j\in\mathbb{Z}_p[[z]]$ and we obtain, $$\widetilde{\omega}(z^p)\frac{\alpha_0}{p}+(\delta\widetilde{\omega})(z^p)\frac{s_2}{\widetilde{y}_0}+\cdots+(\delta^{n-1}\widetilde{\omega})(z^p)\frac{s_{n}}{\widetilde{y}_0}=\alpha_0\widetilde{\omega}+\beta.$$
As $\bm{r}(L)\geq1$ then, by Theorem~\ref{theo_generic_integral}, $y_0=\mathfrak{f}(z)\in 1+z\mathbb{Z}_p[[z]]$. Thus, $\widetilde{y_0}\in1+z\mathbb{Z}_p[[z]]$ and, for every $j\in\{2,\ldots,n\}$, $\gamma_j:=s_j/\widetilde{y}_0\in\mathbb{Z}_p[[z]]$. We put $\bm\omega=\delta\widetilde{\omega}$. Then, 
 \begin{equation*}
 \widetilde{\omega}(z^p)\frac{\alpha_0}{p}+\bm\omega(z^p)\gamma_2+\cdots+(\delta^{n-2}\bm\omega)(z^p)\gamma_{n}=\alpha_0\widetilde{\omega}+\beta.
 \end{equation*}
 By applying $\delta$ to this equality we obtain $$\delta(\widetilde{\omega}(z^p))\frac{\alpha_0}{p}+\delta(\bm\omega(z^p)\gamma_2)+\cdots+\delta((\delta^{n-2}\bm\omega)(z^p)\gamma_{n})=\alpha_0\bm\omega.$$
 As $\delta(\widetilde{\omega}(z^p))=p(\delta\widetilde\omega)(z^p)$ then $\delta(\widetilde{\omega}(z^p))\frac{1}{p}=\bm\omega(z^p)$. Therefore, $$\alpha_0\bm\omega(z^p)+\delta(\bm\omega(z^p)\gamma_2)+\cdots+\delta((\delta^{n-2}\bm\omega)(z^p)\gamma_{n})=\alpha_0\bm\omega.$$
 Now, $\delta((\delta^{i}\bm\omega)(z^p)\gamma_{i+2})=p(\delta^{i+1}\bm\omega)(z^p)\gamma_{i+2}+(\delta^{i}\bm\omega)(z^p)\delta(\gamma_{i+2})$. Further, $\alpha_0\in\mathbb{Z}_p^*$ since $\alpha_0=t_{1,1}(0)=\phi_{1,1}(0)$ and by hypotheses, $|\phi_{1,1}(0)|=1$. Thus, there are $a_0,\ldots, a_n\in\mathbb{Z}_p[[z]]$ such that $$a_0\bm\omega(z^p)+a_1(\delta\bm\omega)(z^p)+\cdots+a_{n-1}(\delta^{n-1}\bm\omega)(z^p)=\bm\omega.$$
 Finally, we set $\mathfrak{A}:=a_0+\frac{a_1}{p}\delta+\cdots+\frac{a_{n-1}}{p^{n-1}}\delta^{n-1}$. So, $\mathfrak{A}(\bm\omega(z^p))=\bm\omega$. Therefore, from Lemma~\ref{prop_p_int}, we obtain $\delta\widetilde{\omega}=\bm\omega\in 1+z\mathbb{Z}_p[[z]]$.
 \end{proof}
 
  \begin{prop}\label{prop_delta}
Let $L$ be a MUM differential operator of order $n$ with coefficients in $\mathbb{Z}_p[[z]]$. Let $y_0=\mathfrak{f}(z)$ and  $y_1=\mathfrak{f}(z)\log z+\mathfrak{g}(z)$ be the solutions of $L$ with $\mathfrak{f}(z)\in 1+z\mathbb{Q}_p[z]]$ and $\mathfrak{g}(z)\in z\mathbb{Q}_p[[z]]$. Suppose that $L$ has a $p$\nobreakdash-integral Frobenius structure given by the matrix $\Phi=(\phi_{i,j})_{1\leq i,j\leq n}$. If $|\phi_{1,1}(0)|=1$ then $\delta\omega\in 1+z\mathbb{Z}_p[[z]]$, where $\omega=y_1/y_0$. 
 \end{prop}
 
 \begin{proof}
Let $Y_Lz^{N}$ be the fundamental matrix of solutions of $\delta X=AX$, where $A$ is the companion matrix of $L$. Since $L$ has a $p$\nobreakdash-integral Frobenius structure, by Proposition~\ref{prop_p_integral_frob}, we have $\bm{r}(L)\geq1$. As  $L\in\mathbb{Z}_p[[z]][\delta]$ is MUM and $\bm{r}(L)\geq1$ we can apply  Lemma~\ref{lemm_paso_base} and thus, from (a') of Lemma~\ref{lemm_paso_base}, we conclude that there is a MUM differential operator $L_1$ such that $\widetilde{y}_0=\Lambda_p(\mathfrak{f})$ and $\widetilde{y}_1=\Lambda_p(\mathfrak{f})\log z+p\Lambda_p(\mathfrak{g})$ are solutions of $L_1$. By (b') of Lemma~\ref{lemm_paso_base}, we know that $H_1=Y_L(z)(\Lambda_p(Y_L)(z^{p}))^{-1}\mathrm{diag}(1,p,p^2,\ldots, p^{n-1})$ belongs to $M_n(\mathbb{Z}_p[[z]])$ and $\delta H_1=AH_1-pH_1B_1(z^p)$, where $B_1$ is the companion matrix of $L_1$. Additionally, from Lemma~\ref{lemm_paso_0}, we have $Y_L(z)(\Lambda_p(Y_L)(z^{p}))^{-1}\in GL_n(\mathbb{Z}_p[[z]])$. So, if we put $H_1=(w_{i,j}(z))_{1\leq i,j\leq n}$ then $w_{1,j}(z)\in p^{j-1}\mathbb{Z}_p[[z]]$ for all $1\leq j\leq n$. Since $ H_1\in M_n(\mathbb{Z}_p[[z]])$ and $H_1(0)=\mathrm{diag}(1,p,\ldots, p^{n-1})$, from Lemma~\ref{prop_eig}, we get $$\widetilde{\omega}(z^p)y_0+(\delta\widetilde{\omega})(z^p)r_2+\cdots+(\delta^{n-1}\widetilde{\omega})(z^p)r_{n}=py_0\omega,$$
where $\widetilde{\omega}=\widetilde{y}_1/\widetilde{y}_0$, $r_j=\sum_{i=j}^n\binom{i-1}{i-j}w_{1,i}(z)(\delta^{i-j}\widetilde{y}_0)(z^p)$.
Since $w_{1,2},\ldots, w_{1,n}\in p\mathbb{Z}_p[[z]]$, we obtain $r_2,\ldots, r_{n}\in p\mathbb{Z}_p[[z]]$. So, $r_j=ps_j$ with $s_j\in\mathbb{Z}_p[[z]]$. For this reason, $$\widetilde{\omega}(z^p)\frac{1}{p}+(\delta\widetilde{\omega})(z^p)\frac{s_2}{y_0}+\cdots+(\delta^{n-1}\widetilde{\omega})(z^p)\frac{s_{n}}{y_0}=\omega.$$
We know that $\bm{r}(L)\geq1$ and thus, by Theorem~\ref{theo_generic_integral}, $y_0=\mathfrak{f}(z)\in 1+z\mathbb{Z}_p[[z]]$. So, $1/y_0\in1+z\mathbb{Z}_p[[z]]$ and, for every $j\in\{1,\ldots,n-1\}$, $t_j:=s_j/y_0$ belongs to $\mathbb{Z}_p[[z]]$. We put $\bm\omega=\delta\widetilde{\omega}$. Then, 
 \begin{equation*}
 \widetilde{\omega}(z^p)\frac{1}{p}+(\bm\omega)(z^p)t_2+\cdots+(\delta^{n-2}\bm\omega)(z^p)t_{n}=\omega.
 \end{equation*}
 By applying $\delta$ to this equality we obtain $$\delta(\widetilde{\omega}(z^p))\frac{1}{p}+\delta((\bm\omega)(z^p)t_2)+\cdots+\delta((\delta^{n-2}\bm\omega)(z^p)t_{n})=\delta\omega.$$
 As $\delta(\widetilde{\omega}(z^p))=p(\delta\widetilde\omega)(z^p)$ then $\delta(\widetilde{\omega}(z^p))\frac{1}{p}=\bm\omega(z^p)$. Therefore, $$\bm\omega(z^p)+\delta((\bm\omega)(z^p)t_1)+\cdots+\delta((\delta^{n-2}\bm\omega)(z^p)t_{n-1})=\delta\omega.$$
According to Proposition~\ref{prop_lambda}, $\delta\widetilde{\omega}=\bm\omega$ belongs to $1+z\mathbb{Z}_p[[z]]$. Therefore, $\delta\omega$ belongs to $1+z\mathbb{Z}_p[[z]]$.

 \end{proof}
In order to prove Theorem~\ref{prop_arith}, we recall the following classical lemma.

\begin{lemm}[Dieudonné-Dowrk's Lemma]
Let $p$ be a prime number and let $f(z)$ be in $1+z\mathbb{Q}_p[[z]]$. Then $f(z)\in 1+z\mathbb{Z}_p[[z]]$ if and only if $\frac{f(z)^p}{f(z^p)}\in 1+pz\mathbb{Z}_p[[z]]$.
\end{lemm}

For a proof of this lemma we refer the reader to \cite[Chap. II, Theorem 5.2]{Dworkgfunciones}. As a corollary we have 

\begin{coro}\label{coro_exp}
Let $p$ be a prime number and let $f(z)$ be in $z\mathbb{Q}_p[[z]]$. Then $\exp(f(z))\in 1+z\mathbb{Z}_p[[z]]$ if and only if  $\frac{1}{p}f(z^p)-f(z)\in z\mathbb{Z}_p[[z]]$.
\end{coro}



\begin{theo}\label{prop_arith}
Let $L$ be a MUM differential operator of order $n$ with coefficients in $\mathbb{Z}_p[[z]]$. Let $y_0=\mathfrak{f}(z)$ and  $y_1=\mathfrak{f}(z)\log z+\mathfrak{g}(z)$ be the solutions of $L$ with $\mathfrak{f}(z)\in 1+z\mathbb{Q}_p[z]]$ and $\mathfrak{g}(z)\in z\mathbb{Q}_p[[z]]$. Suppose that $L$ has a $p$\nobreakdash-integral Frobenius structure given by the matrix $\Phi=(\phi_{i,j})_{1\leq i,j\leq n}$. If $|\phi_{1,1}(0)|=1$ then $\exp(y_1/y_0)^p\in\mathbb{Z}_p[[z]].$
\end{theo}

\begin{proof}
We set  $\mathfrak{h}=p\frac{\mathfrak{g}(z)}{\mathfrak{f}(z)}$. Since $\exp(y_1/y_0)^p=z\exp(\mathfrak{h})$, according to Corollary~\ref{coro_exp}, it is sufficient to show that $\frac{1}{p}\mathfrak{h}(z^p)-\mathfrak{h}(z)\in\mathbb{Z}_p[[z]]$. By assumption, there is $\Phi=(\phi_{i,j}(z))_{1\leq i,j\leq n}\in M_n(\mathbb{Z}_p[[z]])$ such that $\delta\Phi=A\Phi-p\Phi A(z^p)$. We put $\omega=\frac{y_1}{y_0}$, where $y_1=\mathfrak{f}(z)\log(z)+\mathfrak{g}(z)$ and $y_0=\mathfrak{f}(z)$. So, it follows from Lemma~\ref{prop_eig} that $$\omega(z^p)\alpha_0y_0+(\delta\omega)(z^p)r_2+\cdots+(\delta^{n-1}\omega)(z^p)r_{n}=y_0(p\alpha_0\omega+\alpha_1),$$
where $r_j=\sum_{i=j}^n\binom{i-1}{i-j}\phi_{1,i}(z)(\delta^{i-j}y_0)(z^p)$, $\alpha_0=\phi_{1,1}(0)$, and $\alpha_1=\phi_{1,2}(0)$.

Consequently, $$\omega(z^p)\alpha_0+(\delta\omega)(z^p)\frac{r_2}{y_0}+\cdots+(\delta^{n-1}\omega)(z^p)\frac{r_{n}}{y_0}=p\alpha_0\omega+\alpha_1.$$
By Proposition~\ref{prop_p_integral_frob} and Theorem~\ref{theo_generic_integral}, we know that $y_0=\mathfrak{f}(z)\in 1+z\mathbb{Z}_p[[z]]$. So, $1/y_0\in1+z\mathbb{Z}_p[[z]]$. Then, for every $j\in\{2,\ldots,n\}$, $r_j/y_0$ belongs to $\mathbb{Z}_p[[z]]$. As $|\phi_{1,1}(0)|=1$ then, by Proposition~\ref{prop_delta}, we have $\delta\omega\in 1+z\mathbb{Z}_p[[z]]$. Thus, we obtain $\omega(z^p)\alpha_0-p\alpha_0\omega\in z\mathbb{Z}_p[[z]]$. Further,  $\alpha_0\in\mathbb{Z}_p^{*}$ because $\alpha_0=\phi_{1,1}(0)$ and $|\phi_{1,1}(0)|=1$. Hence, $\omega(z^p)-p\omega\in z\mathbb{Z}_p[[z]]$. But $$\omega(z^p)-p\omega=\left[\frac{p\mathfrak{f}(z^p)\log z+\mathfrak{g}(z^p)}{\mathfrak{f}(z^p)}\right]-\left[\frac{p\mathfrak{f}(z)\log z+p\mathfrak{g}(z)}{\mathfrak{f}(z)}\right]=\frac{\mathfrak{g}(z^p)}{\mathfrak{f}(z^p)}-p\frac{\mathfrak{g}(z)}{\mathfrak{f}(z)}=\frac{1}{p}\mathfrak{h}(z^p)-\mathfrak{h}(z).$$
 \end{proof}
\section{Proof of Theorem~\ref{theo_integral_mirror}}\label{sec_proof}
Given that, by assumption $L$ has a $p$-integral Frobenius structure, it follows from Proposition~\ref{prop_p_integral_frob} that $\bm{r}(L)\geq1$. Thus, by Theorem~\ref{theo_generic_integral}, we get $y_0(z)\in 1+z\mathbb{Z}_p[[z]]$. Let us now prove (1) and (2).

(1).  By definition $L^{(2)}=(\delta-t_2)(\delta-t_1)$, where 

$$t_1=\frac{\delta\mathfrak{f}(z)}{\mathfrak{f}(z)},\quad\text{ }t_2=\frac{\delta\mathfrak{h}(z)}{\mathfrak{h}(z)}\quad\text{ with }\mathfrak{h}(z)=\mathfrak{f}(z)+\delta(\mathfrak{g}(z))-t_1\mathfrak{g}(z).$$

As $|\phi_{1,1}(0)|=1$ then, by Proposition~\ref{prop_delta},  $\delta(y_1/y_0)\in 1+z\mathbb{Z}_p[[z]]$. In particular, $\delta(\mathfrak{g}(z)/\mathfrak{f}(z))$ belongs to $\mathbb{Z}_p[[z]]$. Since $$\delta\left(\frac{\mathfrak{g}(z)}{\mathfrak{f}(z)}\right)=\frac{(\delta\mathfrak{g}(z))\mathfrak{f}(z)-\mathfrak{g}(z)(\delta\mathfrak{f}(z))}{\mathfrak{f}(z)^2}$$
and $\mathfrak{f}(z)\in 1+z\mathbb{Z}_p[[z]]$, we obtain $(\delta\mathfrak{g}(z))\mathfrak{f}(z)-\mathfrak{g}(z)(\delta\mathfrak{f}(z))\in\mathbb{Z}_p[[z]]$. Notice that $$\mathfrak{h}(z)=\mathfrak{f}(z)+\delta(\mathfrak{g}(z))-t_1\mathfrak{g}(z)=\mathfrak{f}(z)+\delta(\mathfrak{g}(z))-\frac{\delta\mathfrak{f}(z)}{\mathfrak{f}(z)}\mathfrak{g}(z)=\frac{\mathfrak{f}(z)^2+(\delta\mathfrak{g}(z))\mathfrak{f}(z)-(\delta\mathfrak{f}(z))\mathfrak{g}(z)}{\mathfrak{f}(z)}.$$
So, $\mathfrak{h}(z)\in 1+z\mathbb{Z}_p[[z]]$. Consequently, $t_1$ and $t_2$ belong to $z\mathbb{Z}_p[[z]]$. Hence,  $L^{(2)}\in\mathbb{Z}_p[[z]][\delta]$ and it is clear that $L^{(2)}$ is MUM type.

We now proceed to show that $L^{(2)}$ has a $p$-adic Frobenius structure. We set 
\begin{equation}\label{eq_psi}
\Theta=J(z)\mathrm{diag}(1,p)J(z^p)^{-1},\text{ with } J(z)=\begin{pmatrix} 
\mathfrak{f}(z) & \mathfrak{g}(z)\\
\delta\mathfrak{f}(z) & \mathfrak{f}(z)+\delta\mathfrak{g}(z)\\
\end{pmatrix}
\end{equation}

First, we prove that $\delta\Theta=A_2\Theta-p\Theta A_2(z^p)$, where $A_2$ is the companion matrix of $L^{(2)}$. Notice that $$\Theta=J(z)\begin{pmatrix} 
1 & \log z\\
0 & 1\\
\end{pmatrix}\mathrm{diag}(1,p)\begin{pmatrix} 
1 & -p\log z\\
0 & 1\\
\end{pmatrix}J(z^p)^{-1}.$$
So, 
\begin{equation}\label{eq_mat_eq}
\Theta J(z^p)\begin{pmatrix} 
1 & p\log z\\
0 & 1\\
\end{pmatrix}=J(z)\begin{pmatrix} 
1 & \log z\\
0 & 1\\
\end{pmatrix}\mathrm{diag}(1,p).
\end{equation}

It is clear that the matrix $J(z)\begin{pmatrix} 
1 & \log z\\
0 & 1\\
\end{pmatrix}\mathrm{diag}(1,p)$ is a solution of the system $\delta X=A_2X$ and that $ J(z^p)\begin{pmatrix} 
1 & p\log z\\
0 & 1\\
\end{pmatrix}$ is a solution of the system $\delta X=pA_2(z^p)X$. Consequently, we deduce from~\eqref{eq_mat_eq} that 
\begin{equation}\label{eq_frob_str}
\delta\Theta=A_2\Theta-p\Theta A_2(z^p).
\end{equation}
By \eqref{eq_psi}, we have 
\begin{equation*}
\Theta=\begin{pmatrix} 
\theta_{1,1}(z) & \theta_{1,2}(z)\\
\theta_{2,1}(z) & \theta_{2,2}(z)\\
\end{pmatrix},
\end{equation*}
where 
$$\theta_{1,1}(z)=\frac{\mathfrak{f}(z)((\mathfrak{f}+\delta\mathfrak{g})(z^p))-p\mathfrak{g}(z)((\delta\mathfrak{f})(z^p))}{\Delta},\text{ } \theta_{1,2}(z)=\frac{p\mathfrak{g}(z)(\mathfrak{f}(z^p))-\mathfrak{f}(z)(\mathfrak{g}(z^p))}{\Delta},$$

$$\theta_{2,1}(z)=\frac{(\delta\mathfrak{f}(z))((\mathfrak{f}+\delta\mathfrak{g})(z^p))-p(\mathfrak{f}(z)+\delta\mathfrak{g}(z))(\delta\mathfrak{f}(z^p))}{\Delta},$$ 

$$\theta_{2,2}(z)=\frac{p(\mathfrak{f}(z)+\delta\mathfrak{g}(z))(\mathfrak{f}(z^p))-(\delta\mathfrak{f}(z))(\mathfrak{g}(z^p))}{\Delta}.$$

with $\Delta=(\mathfrak{f}(\mathfrak{f}+\delta\mathfrak{g})-\mathfrak{g}(\delta\mathfrak{f}))(z^p)$. We next prove that $\theta_{1,1}(z),\theta_{1,2}(z)\in\mathbb{Z}_p[[z]].$

Note that $\Delta=(\mathfrak{h}\mathfrak{f})(z^p)$. As $\mathfrak{h}(z),\mathfrak{f}(z)\in 1+z\mathbb{Z}_p[[z]]$ then $\Delta\in 1+z\mathbb{Z}_p[[z]]$. From Theorem~\ref{prop_arith}, we have $\exp(y_1/y_0)^p\in\mathbb{Z}_p[[z]]$ and, by Corollary~\ref{coro_exp}, that is equivalent to saying that $\frac{\mathfrak{g}(z^p)}{\mathfrak{f}(z^p)}-p\frac{\mathfrak{g}(z)}{\mathfrak{f}(z)}\in z\mathbb{Z}_p[[z]]$. Since $\mathfrak{f}(z)\in 1+z\mathbb{Z}_p[[z]]$, we obtain $p\mathfrak{g}(z)\mathfrak{f}(z^p)-\mathfrak{f}(z)\mathfrak{g}(z^p)\in\mathbb{Z}_p[[z]]$. Thus, $\theta_{1,2}\in z\mathbb{Z}_p[[z]]$. Furthermore, it is clear from Equation~\eqref{eq_frob_str} that $$\theta_{1,1}(z)(\mathfrak{f}(z^p))+\theta_{1,2}(z)((\delta\mathfrak{f})(z^p))=\mathfrak{f}(z).$$
In particular, $$\theta_{1,1}(z)=\frac{\mathfrak{f}(z)-\theta_{1,2}(z)((\delta\mathfrak{f})(z^p))}{\mathfrak{f}(z^p)}.$$
We have already seen that $\theta_{1,2}(z)\in\mathbb{Z}_p[[z]]$ and we know that $\mathfrak{f}(z)\in 1+z\mathbb{Z}_p[[z]]$. Therefore, $\theta_{1,1}(z)\in 1+z\mathbb{Z}_p[[z]]$.

Let us write $L^{(2)}=\delta^2+a_1(z)\delta+a_2(z)$ with $a_1(z), a_2(z)\in z\mathbb{Z}_p[[z]]$. We put \begin{equation*}
\Psi=\begin{pmatrix} 
\psi_{1,1}(z) & \psi_{1,2}(z)\\
\psi_{2,1}(z) & \psi_{2,2}(z)\\
\end{pmatrix}=\begin{pmatrix} 
\theta_{1,1}(z) & \theta_{1,2}(z)\\
\delta\theta_{1,1}(z)-p\theta_{1,2}(z)a_2(z^p) & \delta\theta_{1,2}(z)+p\theta_{1,1}(z)-p\theta_{1,2}(z)a_{1}(z^p)\\
\end{pmatrix}.
\end{equation*}
Consequently, $\Psi\in M_2(\mathbb{Z}_p[[z]])$. We now show that
\begin{equation}\label{eq_frob_psi}
\delta\Psi=A_2\Psi-p\Psi A_2(z^p).
\end{equation}
 Let $y$ be a solution of $L^{(2)}$ in $\mathbb{Q}_p[[z]]+\mathbb{Q}_p[[z]]\log z$. Then, from Equation~\eqref{eq_frob_str}, it follows that $$r=\theta_{1,1}(z)(y(z^p))+\theta_{1,2}(z)((\delta y)(z^p))$$
is solution of $L^{(2)}$. Since, by definition $\theta_{1,1}(z)=\psi_{1,1}(z)$ and $\theta_{1,2}(z)=\psi_{1,2}(z)$, we get 
\begin{equation}\label{eq_psi_1}
r=\psi_{1,1}(z)(y(z^p))+\psi_{1,2}(z)((\delta y)(z^p))
\end{equation}
is solution of $L^{(2)}$. So, by applying $\delta$ to the previous equality and by using the fact $\delta ^2y=-a_1(z)\delta y-a_2(z)y$, we get $$\delta r=\psi_{2,1}(z)(y(z^p))+\psi_{2,2}(z)((\delta y)(z^p)).$$
For this reason, \[\Psi\begin{pmatrix}
y(z^p)\\
(\delta y)(z^p)\end{pmatrix}=\begin{pmatrix}
r\\
\delta r\end{pmatrix}.
\]
Thus, by applying $\delta$ to the last equality, we get  \[\delta\Psi\begin{pmatrix}
y(z^p)\\
(\delta y)(z^p)\end{pmatrix}+\Psi pA_2(z^p)\begin{pmatrix}
y(z^p)\\
(\delta y)(z^p)\end{pmatrix}=A_2\begin{pmatrix}
r\\
\delta r\end{pmatrix}=A_2\Psi\begin{pmatrix}
y(z^p)\\
(\delta y)(z^p)\end{pmatrix}.
\]
Since $y$ is an arbitrary solution of $L^{(2)}$, the last equality implies $\delta\Psi=A_2\Psi-p\Psi A_2(z^p)$.

(2). $(a)\Rightarrow(b)$.  Let us suppose that $\exp(y_1/y_0)\in z\mathbb{Z}_p[[z]]$. According to Corollary~\ref{coro_exp}, that is equivalent to saying that $\frac{1}{p}\frac{\mathfrak{g}(z^p)}{\mathfrak{f}(z^p)}-\frac{\mathfrak{g}(z)}{\mathfrak{f}(z)}\in z\mathbb{Z}_p[[z]]$. By applying Proposition~\ref{prop_eig} to $L^{(2)}$ and $\Psi$, we deduce that 
\begin{equation}\label{eq_final}
\frac{1}{p}\frac{\mathfrak{g}(z^p)}{\mathfrak{f}(z^p)}-\frac{\mathfrak{g}(z)}{\mathfrak{f}(z)}=-(\delta(y_1/y_0))(z^p)\frac{\psi_{1,2}}{p}\frac{\mathfrak{f}(z^p)}{\mathfrak{f}(z)}.
\end{equation}
From Proposition~\ref{prop_delta}, we have $\delta(y_1/y_0)\in 1+z\mathbb{Z}_p[[z]]$ and we also know that $\mathfrak{f}(z)\in 1+z\mathbb{Z}_p[[z]]$. Thus, from Equation~\eqref{eq_final}, we get $\psi_{1,2}(z)\in p\mathbb{Z}_p[[z]]$. In addition, from Equation~\eqref{eq_psi_1} follows that $r=\psi_{1,1}(z)y_0(z^p)+\psi_{1,2}(z)((\delta y_0)(z^p))$ is solution of $L^{(2)}$. Note that $r\in 1+z\mathbb{Q}_p[[z]]$. Thus, according to (2) of Remark~\ref{rem_sol}, we get $$r=\psi_{1,1}(z)y_0(z^p)+\psi_{1,2}(z)((\delta y_0)(z^p))=y_0(z).$$ Whence,  $\psi_{1,1}(z)y_0(z^p)=y_0(z)\bmod p.$

$(b)\Rightarrow(c)$. Suppose that $\psi_{1,1}(z)y_0(z^p)=y_0(z)\bmod p.$ Then, by applying $\Lambda_p$ to the previous equality, we get $\Lambda_p(\psi_{1,1}(z))y_0(z)=\Lambda_p(y_0(z))\bmod p$. 

$(c)\Rightarrow(a)$ Let us assume that $\Lambda_p(\psi_{1,1}(z))=\frac{\Lambda_p(y_0(z))}{y_0(z)}\bmod p$. By construction, $$\psi_{1,1}(z)=\theta_{1,1}(z)=\frac{y_0(z)-\theta_{1,2}(z)(\delta y_0(z))(z^p)}{y_0(z^p)}.$$ Hence $$\Lambda_p(\psi_{1,1}(z))=\frac{\Lambda_p(y_0(z))-(\Lambda_p(\theta_{1,2}(z)))(\delta y_0(z))}{y_0(z)}.$$
Therefore, $$\frac{\Lambda_p(y_0(z))-(\Lambda_p(\theta_{1,2}(z)))(\delta y_0(z))}{y_0(z)}=\frac{\Lambda_p(y_0(z))}{y_0(z)}\bmod p.$$ Whence, we obtain $\Lambda_{p}(\theta_{1,2}(z))\in p\mathbb{Z}_p[[z]]$.  By definition, $\theta_{1,2}(z)=\psi_{1,2}(z)$. Further, it follows from Equation~\eqref{eq_frob_psi} that $\delta\Psi=A_2\Psi\bmod p$. Since $A_2$ is the companion matrix of $L^{(2)}$, we get that $\psi_{1,2}(z)\bmod p$ is solution of $L^{(2)}\bmod p$. Given that $L\in \mathbb{Z}_p[[z]][\delta]L^{(2)}$, we then have $\psi_{1,2}(z)\bmod p$ is solution of $L\bmod p$. Thus, from Lemma 6.6 and Lemma 6.7 of \cite{vargas2}, we deduce that \[\psi_{1,2}(z)=y_0(z)\frac{(\Lambda_p(\psi_{1,2}))(z^p)}{(\Lambda_p(y_0))(z^p)}\bmod p.\] But, we have $\Lambda_p(\psi_{1,2}(z))=0\bmod p$ because  $\theta_{1,2}(z)=\psi_{1,2}(z)$ and $\Lambda_{p}(\theta_{1,2}(z))\in p\mathbb{Z}_p[[z]]$. Thus, $\psi_{1,2}=0\bmod p$. That is, $\theta_{1,2}\in p\mathbb{Z}_p[[z]]$. However, we know that $$\theta_{1,2}(z)=\frac{p\mathfrak{g}(z)(\mathfrak{f}(z^p))-\mathfrak{f}(z)(\mathfrak{g}(z^p))}{\Delta},$$
where $\Delta\in 1+z\mathbb{Z}_p[[z]]$. So $p\mathfrak{g}(z)(\mathfrak{f}(z^p))-\mathfrak{f}(z)(\mathfrak{g}(z^p))\in p\mathbb{Z}_p[[z]]$. Thus, since $\mathfrak{f}(z)\in 1+z\mathbb{Z}_p[[z]]$, we get $\frac{1}{p}\frac{\mathfrak{g}(z^p)}{\mathfrak{f}(z^p)}-\frac{\mathfrak{g}(z)}{\mathfrak{f}(z)}\in z\mathbb{Z}_p[[z]]$. Consequently, Corollary~\ref{coro_exp} implies $\exp(y_1/y_0)\in\mathbb{Z}_p[[z]]$.

$\hfill\square$


\begin{rema}\label{rem_equivalence}
We assume that $L$ satisfies the assumptions of Theorem~\ref{theo_integral_mirror}. In this remark we show that $\exp(y_1/y_0)\in z\mathbb{Z}_p[[z]]$ if and only if there exists $\mathcal{B}=B_{1}(z)+B_{2}(z)\delta\in\mathbb{Z}_p[[z]]$ with $B_1(0)=1$ and $B_2(0)=0$ such that, for every solution $y$ of $L^{(2)}$, $\mathcal{B}(y(z^p))$ is solution of $L^{(2)}$. Indeed, according to (1) of Theorem~\ref{theo_integral_mirror},  $L^{(2)}$ has a $p$-integral Frobenius structure $\Psi=(\psi_{i,j}(z))_{1\leq i,j\leq2}$ such that $\Psi(0)=\mathrm{diag}(1,p)$. Let us suppose that $\exp(y_1/y_0)\in z\mathbb{Z}_p[[z]]$. Thus, by following the proof of $(a)\Rightarrow(b)$, we know that $\psi_{1,2}(z)=pt_{1,2}(z)$ with $t_{1,2}(z)\in\mathbb{Z}_p[[z]]$. Moreover, $\psi_{1,2}(z)((\delta y)(z^p))=t_{1,2}(z)\delta(y(z^p)).$ Therefore, if we set $\mathcal{B}=\psi_{1,1}(z)+t_{1,2}(z)\delta$, from Equation~\eqref{eq_psi_1}, we get that, for every  $y$ of $L^{(2)}$, $$\mathcal{B}(y(z^p))=\psi_{1,1}(z)(y(z^p))+t_{1,2}(z)\delta(y(z^p))=\psi_{1,1}(z)(y(z^p))+\psi_{1,2}(z)((\delta y)(z^p))$$
is solution of $L^{(2)}$. Further, by construction $\psi_{1,1}(0)=1$ and $t_{1,2}(0)=0$ because $\psi_{1,2}(0)=0$.

We now assume that there exists $\mathcal{B}=B_{1}(z)+B_{2}(z)\delta\in\mathbb{Z}_p[[z]]$ with $B_1(0)=1$ and $B_2(0)=0$ such that, for every solution $y$ of $L^{(2)}$, $\mathcal{B}(y(z^p))$ is solution of $L^{(2)}$. Therefore, for every solution $y$ of $L^{(2)}$, $$B_1(z)y(z^p)+pB_2(z)((\delta y)(z^p))$$
is solution of $L^{(2)}$. Let us write $L^{(2)}=\delta^2+a_1(z)\delta+a_2(z)$. According to (1) of Theorem~\ref{theo_integral_mirror}, $a_1(z), a_2(z)\in z\mathbb{Z}_p[[z]]$. We now put \[\Gamma=\begin{pmatrix} 
B_{1}(z) & pB_{2}(z)\\
\delta B_{1}(z)-p^2B_{2}(z)a_2(z^p) & \delta B_{2}(z)+pB_{1}(z)-p^2B_{2}(z)a_{1}(z^p)\\
\end{pmatrix}.\]
Then, $\Gamma\in M_4(\mathbb{Z}_p[[z]])$ and it is not hard to see that $\Gamma=A_2\Gamma-p\Gamma A_2(z^p)$. Hence, following Remark~\ref{rem_det}, we have $\Gamma=Y_{L^{(2)}}\Gamma(0)(Y_{L^{(2)}}(z^p))^{-1}$. It is clear that $\Gamma(0)=\mathrm{diag}(1,p)$. We already know that $\Psi$ is a $p$-integral Frobenius structure for $L^{(2)}$ and hence, by Remark~\ref{rem_det} again, $\Psi=Y_{L^{(2)}}\Psi(0)(Y_{L^{(2)}}(z^p))^{-1}$ and we know that $\Psi(0)=\mathrm{diag}(1,p)$. So, $\Gamma=\Psi$ and thus $\psi_{1,2}\in p\mathbb{Z}_p[[z]]$. Further,  from Equation~\eqref{eq_psi_1} follows that $r=\psi_{1,1}(z)y_0(z^p)+\psi_{1,2}(z)((\delta y_0)(z^p))$ is solution of $L^{(2)}$. Note that $r\in 1+z\mathbb{Q}_p[[z]]$. Thus, according to (2) of Remark~\ref{rem_sol}, we get $$r=\psi_{1,1}(z)y_0(z^p)+\psi_{1,2}(z)((\delta y_0)(z^p))=y_0(z).$$Consequently, $\psi_{1,1}(z)y_0(z^p)=y_0(z)\bmod p$. Consequently, according to Theorem~\ref{theo_integral_mirror}, $\exp(y_1/y_0)\in z\mathbb{Z}_p[[z]]$.
\end{rema}

\end{document}